%% file: nonabelian_basechange_etale_homotopy.tex
\pdfoutput=1
\documentclass[leqno]{article}
\usepackage[normal]{phaine_style}
\usepackage[margin=1.5in]{geometry}
\input{document_macros}
\title{\Large Nonabelian basechange theorems \& étale homotopy theory}

\author{\normalsize Peter J. Haine \and\normalsize Tim Holzschuh \and\normalsize Sebastian Wolf}

\date{\normalsize \today}

\begin{document}

\maketitle

\begin{abstract} 
	This paper has two main goals.
	First, we prove nonabelian refinements of basechange theorems in étale cohomology (i.e., prove analogues of the classical statements for sheaves of \textit{spaces}).
	Second, we apply these theorems to prove a number of results about the étale homotopy type.
	Specifically, we prove nonabelian refinements of the smooth basechange theorem, Huber--Gabber affine analogue of the proper basechange theorem, and Fujiwara--Gabber rigidity theorem.
	Our methods also recover Chough's nonabelian refinement of the proper basechange theorem.
	Transporting an argument of Bhatt--Mathew to the nonabelian setting, we apply nonabelian proper basechange to show that the profinite étale homotopy type satisfies \arcdescent.
	Using nonabelian smooth and proper basechange and descent, we give rather soft proofs of a number of Künneth formulas for the étale homotopy type.
\end{abstract}

\setcounter{tocdepth}{1}

\tableofcontents


\setcounter{section}{-1}
\input{content/introduction}


\input{content/acknowledgments}


\input{content/background}


\input{content/classical_basechange_to_nonabelian_basechange}


\input{content/arc-descent}


\input{content/Kunneth_formulas}


\DeclareFieldFormat{labelnumberwidth}{#1}
\printbibliography[keyword=alph, heading=references]
\DeclareFieldFormat{labelnumberwidth}{{#1\adddot\midsentence}}
\printbibliography[heading=none, notkeyword=alph]

\end{document}

%% file: document_macros.tex
\newcommand{\from}{\colon}

\theoremstyle{definition}
\newtheorem*{proofoverview}{Proof Overview}

\newcommand{\aicvexcision}{\xspace{aic-v-ex\-ci\-sion}\xspace}

\newcommand{\arccosheaf}{\xspace{arc-co\-sheaf}\xspace}
\newcommand{\arccover}{\xspace{arc-cover}\xspace}
\newcommand{\arccovers}{\xspace{arc-covers}\xspace}

\newcommand{\archypercovering}{\xspace{arc-hy\-per\-cover\-ing}\xspace}
\newcommand{\archyperdescent}{\xspace{arc-hy\-per\-de\-scent}\xspace}
\newcommand{\archypersheaf}{\xspace{arc-hy\-per\-sheaf}\xspace}
\newcommand{\arcdescent}{\xspace{arc-de\-scent}\xspace}
\newcommand{\arcsheaf}{\xspace{arc-sheaf}\xspace}
\newcommand{\arcsheaves}{\xspace{arc-sheaves}\xspace}
\newcommand{\arctopology}{\xspace{arc-topol\-o\-gy}\xspace}

\newcommand{\htopology}{\xspace{h-topol\-o\-gy}\xspace}

\newcommand{\vcosheaf}{\xspace{v-co\-sheaf}\xspace}
\newcommand{\vcover}{\xspace{v-cover}\xspace}

\newcommand{\vhypercover}{\xspace{v-hy\-per\-cover}\xspace}
\newcommand{\vhypercovers}{\xspace{v-hy\-per\-covers}\xspace}
\newcommand{\vhyperdescent}{\xspace{v-hy\-per\-de\-scent}\xspace}
\newcommand{\vdescent}{\xspace{v-de\-scent}\xspace}
\newcommand{\vsheaf}{\xspace{v-sheaf}\xspace}

\newcommand{\vtopology}{\xspace{v-topol\-o\-gy}\xspace}

\newcommand{\prosmooth}{\xspace{pro\-smooth}\xspace}

\newcommand{\pprimecomplete}{\xspace{$p'$-com\-plete}\xspace}

\newcommand{\pclosed}{\xspace{$p$-closed}\xspace}

\newcommand{\pprimeprofinite}{\xspace{$p'$-pro\-fi\-nite}\xspace}

\newcommand{\Sigmaprimeclosed}{\xspace{$\Sigma'$-closed}\xspace}

\newcommand{\Sigmacomplete}{\xspace{$\Sigma$-com\-plete}\xspace}
\newcommand{\Sigmacompletion}{\xspace{$\Sigma$-com\-ple\-tion}\xspace}

\newcommand{\Sigmagroup}{\xspace{$\Sigma$-group}\xspace}
\newcommand{\Sigmagroups}{\xspace{$\Sigma$-groups}\xspace}
\newcommand{\Sigmafinite}{\xspace{$\Sigma$-fi\-nite}\xspace}
\newcommand{\Sigmaprofinite}{\xspace{$\Sigma$-pro\-fi\-nite}\xspace}
\newcommand{\Sigmatorsion}{\xspace{$\Sigma$-tor\-sion}\xspace}

\newcommand{\proSigma}{\xspace{pro-$\Sigma$}\xspace}

\newcommand{\Ubf}{\mathbfit{U}}

\newcommand{\longfromto}[2]{{#1}\longrightarrow{#2}}

\newcommand{\RTop}{\categ{RTop}}

\DeclareMathOperator{\EM}{EM}
\DeclareMathOperator{\Abobj}{Ab}
\DeclareMathOperator{\Grpobj}{Grp}

\DeclareMathOperator{\lisse}{lis}

\DeclareMathOperator{\inj}{inj}
\newcommand{\Deltainj}{\DDelta_{\inj}}
\newcommand{\Deltainjop}{\Deltainj^{\op}}

\DeclareMathOperator{\acc}{acc}
\newcommand{\Funlexacc}{\Fun^{\lex,\acc}}

\renewcommand{\Stab}{\mathrm{Sp}}

\newcommand{\cospecializes}{\leftsquigarrow}

\newcommand{\pprimecomp}{_{p'}^{\wedge}}

\newcommand{\piet}{\uppi^{\et}}

\newcommand{\Gk}{\Gup_{k}}
\newcommand{\BGk}{\mathrm{BG}_{k}}
\newcommand{\BG}{\mathrm{BG}}

\newcommand{\Piet}{\Pi_{\infty}^{\et}}
\newcommand{\Pietprotrun}{\Pi_{<\infty}^{\et}}
\newcommand{\Pietprofin}{\widehat{\Pi}_{\infty}^{\et}}

\DeclareMathOperator{\qcqs}{qcqs}
\newcommand{\Schqcqs}{\Sch^{\qcqs}}
\newcommand{\Schqcqsop}{\Sch^{\qcqs,\op}}

\DeclareMathOperator{\spmap}{sp}

\DeclareMathOperator{\characteristic}{char}

\newcommand{\ProCat}{\Pro(\Catinfty)}

\newcommand{\ProSpc}{\Pro(\Spc)}
\newcommand{\Spcfin}{\Space_{\uppi}}
\newcommand{\Spctrun}{\Space_{<\infty}}
\newcommand{\ProSpcfin}{\Pro(\Spcfin)}
\newcommand{\ProSpctrun}{\Pro(\Spctrun)}

\newcommand{\SpcSigma}{\Spc_{\Sigma}}
\newcommand{\ProSpcSigma}{\Pro(\SpcSigma)}
\newcommand{\Sigmacomp}{_{\Sigma}^{\wedge}}

\newcommand{\Spcpprime}{\Spc_{p'}}
\newcommand{\ProSpcpprime}{\Pro(\Spcpprime)}

\newcommand{\Gammaet}{\Gamma_{\!\et}}
\newcommand{\Gammauppersharp}{\Gamma^{\sharp}}

\DeclareMathOperator{\alg}{alg}
\newcommand{\kalg}{k^{\alg}}

\newcommand{\Xkbar}{X_{\kbar}}
\newcommand{\Xs}{X_{s}}

\newcommand{\Ykbar}{Y_{\kbar}}

\newcommand{\Loc}[2]{{#1}_{(#2)}}
\newcommand{\Locet}[2]{{#1}_{(#2),\et}}

\newcommand{\Het}{\Hup_{\et}}

\DeclareMathOperator{\Sym}{Sym}

\newcommand{\Osh}{\Ocal^{\sh}}

\newcommand{\Icomp}{^{\wedge}}

\newcommand{\Dcons}{\Dup_{\cons}}

\renewcommand{\Uhat}{U^{\wedge}}

\newcommand{\fuppersharp}{f^{\sharp}}

\DeclareMathOperator{\map}{map}

%% file: content/introduction.tex

\section{Introduction}\label{sec:introduction}

This paper has two central themes.
First, we prove nonabelian refinements of essentially all basechange theorems in étale cohomology.
More precisely, basechange theorems in étale cohomology are usually proven for sheaves of sets or abelian groups; we explain how to generalize these results to sheaves valued in the \category of \textit{spaces}.

Second, we apply these nonabelian basechange theorems to give rather soft proofs of a number of results in étale homotopy theory (see \cref{subsec:application_arc-descent,subsec:Kunneth_formuals}).
Often it is technically possible to prove results in étale homotopy theory in two steps by separately proving a result for étale fundamental groups, and then using a basechange theorem for étale cohomology of abelian sheaves.
However, our perspective is that it is actually easier to prove these results directly from the nonabelian refinements of the basechange theorems.
Moreover, we are often able to remove restrictive hypotheses from statements currently available in the literature, as well as prove new results.

We start by explaining the nonabelian basechange theorems that we prove.


\subsection{Nonabelian basechange theorems in algebraic geometry}\label{subsec:nonabelian_basechange_theorems_in_algebraic_geometry}

To demonstrate our approach, let us focus on the nonabelian refinement of the smooth baschange theorem in étale cohomology \cites[Exposé XII, Corollaire 1.2]{MR50:7132}.
For the statement, recall that a morphism of schemes $ f \colon \fromto{X}{Z} $ is \textit{\prosmooth} if $ X $ can be written as the cofiltered limit of smooth $ Z $-schemes with affine transition maps.

\begin{theorem}[(nonabelian smooth basechange; \Cref{cor:nonabelian_smooth_basechange})]\label{intro_thm:nonabelian_smooth_basechange}
	Let
	\begin{equation}\label{sq:general_pullback}
		\begin{tikzcd}
			W \arrow[r, "\fbar"] \arrow[d, "\gbar"'] & Y \arrow[d, "g"] \\ 
			X \arrow[r, "f"'] & Z 
		\end{tikzcd}
	\end{equation}
	be a pullback square of qcqs schemes and assume that the morphism $ f $ is \prosmooth.
	Write $ \Sigma $ for the set of prime numbers invertible on $ Z $.
	Then for each $ \Sigma $-torsion étale sheaf of \emph{spaces} $ F $ on $ Y $ (see \Cref{def:Sigma-torsion_etale_sheaf}), the exchange transformation
	\begin{equation*}
		\fromto{\fupperstar\glowerstar(F)}{\gbarlowerstar\fbarupperstar(F)}
	\end{equation*}
	is an equivalence in the \category of étale sheaves of spaces on $ X $.
\end{theorem}

The assumption that $ F $ is $ \Sigma $-torsion in particular guarantees that there is an integer $ n \geq 0 $ such that the only nonzero homotopy sheaves of $ F $ are in degrees $ \leq n $.
Thus one might hope to prove \Cref{intro_thm:nonabelian_smooth_basechange} by a `Postnikov tower argument' inducting on the truncation degree $ n $.
The idea would be to consider the fibers of the map $ \fromto{F}{\trun_{\leq n-1} F} $ to the $ (n-1) $-truncation of $ F $; these fibers have a homotopy sheaf concentrated in the single degree $ n $.
Basechange for $ \trun_{\leq n-1} F $ is the inductive hypothesis, and basechange for the fibers follows from the classical cohomological basechange.
One might hope that this implies basechange for $ F $.

Unfortunately, there are (at least) two problems with this naive approach.
First, the sheaf $ \trun_{\leq n-1} F $ might not admit a global section, so it is not even clear how to start the inductive step.
That is, it may not even make sense to speak of fibers of the map $ \fromto{F}{\trun_{\leq n-1} F} $.
Second, even if $ F $ admits a global section, the pushforward functors appearing in the exchange transformation $ \fromto{\fupperstar\glowerstar(F)}{\gbarlowerstar\fbarupperstar(F)} $ do not commute with the truncation functors. 

\begin{proofoverview}
	One of the key points of this paper is that, by arguing differently, it \textit{is} possible to reduce \Cref{intro_thm:nonabelian_smooth_basechange} to a claim about basechange for sheaves of $ 1 $-groupoids (i.e., \textit{stacks in groupoids}) and basechange for sheaves of abelian groups. 
	The argument goes roughly as follows.
	First note that in order to show that the exchange morphism $ \fromto{\fupperstar\glowerstar(F)}{\gbarlowerstar\fbarupperstar(F)} $ is an equivalence, it suffices to check the claim after passing to the stalk at each geometric point $ \fromto{x}{X} $.
	We then re-express the stalk of an étale sheaf on $ X $ as the global sections of its pullback to the strict localization $ \Spec(\Ocal_{X,x}^{\sh}) $. 
	Applying an unconditional basechange result about pullbacks along the morphism \smash{$ \fromto{\Spec(\Ocal_{X,x}^{\sh})}{X} $} \cite[Proposition 7.5.1]{arXiv:1807.03281}, we reduce to proving the following: if $ X $ and $ Z $ are spectra of strictly henselian local rings and $ f $ is a prosmooth morphism induced by a local ring homomorphism, then the natural map
	\begin{equation*}
	   \Gammaet(Y; F) \to \Gammaet(W; \fbarupperstar(F)) 
	\end{equation*}
	is an equivalence (see \Cref{cor:when_is_the_exchange_transformation_an_equivalence}).
	Using the theory of $ n $-gerbes, we explain why, for this statement about global sections, it is possible to use a `Postnikov tower argument' to reduce the claim to the cases where $ F $ is a sheaf of $ 1 $-groupoids, and where $ F $ is an Eilenberg--MacLane object (see \Cref{prop:comparing_global_sections,cor:reduction_to_1-truncated_local}).
	The first case was proven by Giraud \cite[Chapitre VII, Théorème 2.1.2]{MR0344253}, and the second case is equivalent to the classical statement for cohomology groups of abelian sheaves.
\end{proofoverview}

This reduction to a claim about global sections of schemes over spectra of strictly henselian local rings works in complete generality.
As a result, using the same method we reprove Chough's nonabelian proper basechange theorem \cite[Theorem 1.2]{MR4493612}, as well as prove nonabelian refinements of the Gabber--Huber affine analogue of the proper basechange theorem \cites{MR1286833}{MR1214956} and the Fujiwara--Gabber rigidity theorem \cite[Corollary 6.6.4]{MR1360610}.
See \cref{subsec:nonabelian_basechange_theorems}.
We also apply the nonabelian smooth and proper basechange theorems to show that, after completion away from the residue characteristics, the étale homotopy types of the geometric fibers of a smooth proper morphism of schemes are invariant under specialization (see \cref{subsec:invariance_under_specialization}).

In the remainder of the introduction, we explain some applications of these nonabelian basechange theorems.


\subsection{Application: \arcdescent}\label{subsec:application_arc-descent}

Bhatt and Mathew recently introduced the \textit{\arctopology} on schemes \cite{MR4278670}.
The \arctopology is finer than the \vtopology, and \arcdescent has a number of useful consequences that do not follow from \vdescent.
For example, \arcsheaves satisfy \textit{Milnor excision} and a version of the Beauville--Laszlo \textit{formal gluing} theorem \cite{MR1320381}.
See \cite[Corollaries 4.25 \& 6.7]{MR4278670}.
Bhatt--Mathew also showed that many familiar étale sheaves satisfy \arcdescent, e.g., étale cohomology with torsion coefficients \cite[Theorem 5.4]{MR4278670}.
The key tool in their proof is the proper basechange theorem.

Once one has access to the nonabelian proper basechange theorem, it is not hard to adjust Bhatt and Mathew's arguments to prove a nonabelian version of this result.
Write $ \ProSpcfin $ for the \category of profinite spaces.
Given a scheme $ X $, write \smash{$ \Pietprofin(X) \in \ProSpcfin $} for the profinite étale homotopy type of $ X $.

\begin{theorem}[(\arcdescent for étale homotopy types; \Cref{thm:Pietprofin_satisfies_arc-descent})]\label{intro_thm:Pietprofin_satisfies_arc-descent}
	The functor
	\begin{equation*}
		\Pietprofin(-) \colon \Schqcqs \to \ProSpcfin
	\end{equation*}
	is a hypercomplete \arccosheaf.
	In other words, for any \archypercovering $ U_\bullet \to X $ the induced morphism
	\begin{equation*}
		\colim_{[n] \in \Deltaop} \Pietprofin(U_n) \to \Pietprofin(X)
	\end{equation*}
	is an equivalence in $ \ProSpcfin $.
\end{theorem}

In the remainder of this subsection, let us explain what Milnor excision and formal gluing mean in the context of étale homotopy theory.
Recall that a commutative square of schemes
\begin{equation}\label{sq:Milnor_square}
	\begin{tikzcd}
		Z \arrow[d] \arrow[r, hooked] & X \arrow[d, "f"] \\ 
		Z' \arrow[r, "i"', hooked] & X'
	\end{tikzcd}
\end{equation}
is a \textit{Milnor square} if it is a pullback square, $ f $ is affine, $ i $ is a closed immersion, and the induced morphism $ \fromto{Z' \coproduct_Z X}{X'} $ is an isomorphism.%
\footnote{By \cite[Théorème 7.1]{MR2044495}, the previous conditions guarantee that the pushout of schemes $ Z' \coproduct_Z X $ exists.}
As a consequence of \arcdescent, we have:

\begin{corollary}[(Milnor excision)]\label{intro_cor:Pietprofin_satisfies_Milnor_excision}
	Given a Milnor square \eqref{intro_cor:Pietprofin_satisfies_Milnor_excision}, the induced square
	\begin{equation*}
		\begin{tikzcd}
			\Pietprofin(Z) \arrow[d] \arrow[r] & \Pietprofin(X) \arrow[d] \\ 
			\Pietprofin(Z') \arrow[r] & \Pietprofin(X')
		\end{tikzcd}
	\end{equation*}
	is a pushout square in $ \ProSpcfin $.
\end{corollary}

Recall that a \textit{formal gluing datum} is a pair $ (\fromto{A}{B},I) $ of a ring homomorphism $ \fromto{A}{B} $ together with a finitely generated ideal $ I \subset A $ such that for each $ n \geq 0 $, we have $ A/I^n \isomorphic B/I^n B $ \cite[Definition 1.14]{MR4278670}.
Again by \arcdescent, we have:

\begin{corollary}[(formal gluing)]
	Given a formal gluing datum $ (\fromto{A}{B},I) $, the induced square
	\begin{equation*}
		\begin{tikzcd}
			\Pietprofin(\Spec(B) \sminus \Vup(IB)) \arrow[d] \arrow[r] & \Pietprofin(\Spec(B)) \arrow[d] \\ 
			\Pietprofin(\Spec(A) \sminus \Vup(I)) \arrow[r] & \Pietprofin(\Spec(A))
		\end{tikzcd}
	\end{equation*}
	is a pushout square in $ \ProSpcfin $.
\end{corollary}


\subsection{Application: Künneth formulas}\label{subsec:Kunneth_formuals}

Let $ k $ be a separably closed field and let $ X $ and $ Y $ be qcqs $ k $-schemes.
Chough observed that if $ Y $ is proper, then the nonabelian proper basechange theorem easily implies that the natural map
\begin{equation}\label{eq:Pietprofin_of_products}
	\Pietprofin(X \cross_k Y) \longrightarrow \Pietprofin(X) \cross \Pietprofin(Y)
\end{equation}
is an equivalence \cite[Theorem 5.3]{MR4493612}.
(On $ \uppi_1 $, this recovers the classical Künneth formula for étale fundamental groups \cites[Exposé X, Corollaire 1.7]{MR50:7129}[Corollary 4.1.23]{Kedlaya:sheaves_stacks_and_shtukas}.)
Similarly, if $ X $ is smooth, then the nonabelian smooth basechange theorem immediately implies that the map \eqref{eq:Pietprofin_of_products} becomes an equivalence after completion away from $ \characteristic(k) $.

We offer two refinements of these results.
First, using the fundamental fiber sequence for étale homotopy types \cite[Corollary 3.21]{arXiv:2209.03476}, we extend Chough's result to arbitrary base fields:

\begin{proposition}[(relative Künneth formula, proper case; \Cref{cor:relative_Kunneth_formula_proper})]\label{intro_prop:relative_Kunneth_formula_proper}
	Let $ k $ be a field with absolute Galois group $ \Gk $, and let $ X $ and $ Y $ be qcqs $ k $-schemes.
	If $ Y $ is proper over $ k $, then the natural map
	\begin{equation*}
		\Pietprofin(X \cross_k Y) \longrightarrow \displaystyle \Pietprofin(X) \crosslimits_{\BGk} \Pietprofin(Y) 
	\end{equation*}
	is an equivalence in $ \ProSpcfin $.
\end{proposition}

\noindent Let $ p $ be a prime number or $ 0 $.
Given a scheme $ X $, write \smash{$ \Piet(X)\pprimecomp $} for the completion of the étale homotopy type of $ X $ at the set of primes \textit{different from} $ p $. 
Using \vdescent, the theory of alterations \cites{Illusie:On_Gabbers_refined_uniformization}[Exposé IX]{MR3309086}{MR1423020}{MR3665001}, and the fundamental fiber sequence, we prove:

\begin{proposition}[(prime-to-$ p $ relative Künneth formula; \Cref{cor:relative_Kunneth_formula_prime-to-p})]\label{intro_prop:relative_Kunneth_formula_prime-to-p}
	Let $ k $ be a field of characteristic $ p \geq 0 $ with absolute Galois group $ \Gk $, and let $ X $ and $ Y $ be qcqs $ k $-schemes.
	If the profinite group $ \Gk $ is prime-to-$ p $, then the natural map 
	\begin{equation*}
		\Piet(X \cross_k Y)\pprimecomp \longrightarrow \displaystyle \Piet(X)\pprimecomp \crosslimits_{\BGk} \Piet(Y)\pprimecomp 
	\end{equation*}
	is an equivalence in $ \ProSpcfin $.
\end{proposition}

\noindent For $ k $ separably closed, \Cref{intro_prop:relative_Kunneth_formula_prime-to-p} recovers a result of Orgogozo \cite[Corollaire 4.9]{MR1975807}.
In addition, \Cref{intro_prop:relative_Kunneth_formula_proper,intro_prop:relative_Kunneth_formula_prime-to-p} imply Künneth formulas for symmetric powers (see \Cref{rem:symmetric_Kunneth_formula_proper_case,rem:symmetric_Kunneth_formula_prime-to-p}).


\subsection*{Linear overview}\label{subsec:linoverview}

\Cref{sec:background} recalls some background about $ n $-gerbes in \topoi and the étale homotopy type; the familiar reader can safely skip this section.
In \cref{sec:reduction_to_local_rings}, we prove nonabelian refinements of: the smooth basechange theorem, the proper basechange theorem, the Gabber--Huber affine analogue of the proper basechange theorem, and the Fujiwara--Gabber rigidity theorem.
See, in particular, \cref{subsec:nonabelian_basechange_theorems}.
We also explain why, after completion away from the residue characteristics, the étale homotopy type of the geometric fibers of a smooth and proper morphism of schemes is invariant under specialization, see \cref{subsec:invariance_under_specialization}.
In \cref{sec:arc-descent} we apply the nonabelian proper basechange theorem to show that the profinite étale homotopy type satisfies \arcdescent.
\Cref{sec:Kunneth_formulas} uses many of the tools developed in the previous sections to prove Künneth formulas for the étale homotopy type.

%% file: content/acknowledgments.tex

\begin{acknowledgments}
	We thank Clark Barwick, Bhargav Bhatt, Chang-Yeon Chough, Denis-Charles Cisinski, Elden Elmanto, and Bogdan Zavyalov for enlightening discussions around the contents of this paper.

	We would like to thank SFB 1085 `Higher Invariants' and the University of Regensburg for its hospitality.
	The first-named author gratefully acknowledges support from the UC President's Postdoctoral Fellowship, NSF Mathematical Sciences Postdoctoral Research Fellowship under Grant \#DMS-2102957, and a grant from the Simons Foundation (816048, LC). 
	The second-named author acknowledges support from the Deutsche Forschungsgemeinschaft (DFG) through the Collaborative Research Centre TRR 326 `Geometry and Arithmetic of Uniformized Structures',  project number 444845124.
	The third-named author was supported by the SFB 1085 `Higher Invariants' in Regensburg, funded by the DFG.
\end{acknowledgments}

%% file: content/background.tex

\section{Background}\label{sec:background}

This section briefly recalls background on constructible and torsion étale sheaves of spaces (\cref{subsec:notation_and_terminology}), gerbes in \topoi (\cref{subsec:gerbes_in_topoi}), the étale homotopy type (\cref{subsec:the_etale_homotopy_type}), and exchange transformations (\cref{subsec:exchange_transformations}).


\subsection{Notation and terminology}\label{subsec:notation_and_terminology}

\begin{notation}
	Given a scheme $ X $, we write $ X_{\et} $ for the \topos of étale sheaves of \textit{spaces} on $ X $.
\end{notation}

\begin{notation}
	Let $ \X $ be \atopos and $ n \geq -2 $ an integer. 
	We write $ \X_{\leq n} \subset \X $ for the full subcategory spanned by the $ n $-truncated objects.
	This inclusion admits a left adjoint that we denote by $ \trun_{\leq n} \colon \fromto{\X}{\X_{\leq n}} $.
\end{notation}

\begin{nul}
	For a scheme $ X $, the subcategory $ X_{\et,\leq 0} \subset X_{\et} $ is the full subcategory spanned by the étale sheaves of \textit{sets} on $ X $.
\end{nul}

\begin{recollection}\label{rec:n-truncation_preserves_filtered_colimits}
	Since filtered colimits commute with finite limits in an \atopos, for any \topos $ \X $ and integer $ n \geq -1 $, the inclusion $ \X_{\leq n} \subset \X $ preserves filtered colimits.
	As a result, the endofunctor $ \trun_{\leq n} \colon \fromto{\X}{\X} $ preserves filtered colimits.
\end{recollection}

Let us now recall the nonabelian refinements of étale sheaves with torsion contained in a set of prime numbers.

\begin{notation}
	Write $ \Spc $ for the \category of (small) spaces and $ \Catinfty $ for the \category of (small) \categories.
\end{notation}

\begin{definition}\label{def:Sigma-finite}
	Let $ \Sigma $ be a set of prime numbers.
	\begin{enumerate}[label=\stlabel{def:Sigma-finite}, ref=\arabic*]
		\item A finite group $ G $ is a \defn{\Sigmagroup} if the order of $ G $ is in the multiplicative closure of $ \Sigma $. 

		\item We say that a space $ K $ is \defn{\pifinite} if $ K $ is truncated, $ \uppi_0(K) $ is finite, and all homotopy groups of $ K $ are finite.
		We write $ \Spcfin \subset \Spc $ for the full subcategory spanned by the \pifinite spaces.

		\item We say that a space $ K $ is \defn{\Sigmafinite} if $ K $ is \pifinite and all homotopy groups of $ K $ are \Sigmagroups.
		We write $ \SpcSigma \subset \Spcfin $ for the full subcategory spanned by the \Sigmafinite spaces.
	\end{enumerate} 
\end{definition}

\begin{notation}
	Let $ \X $ be \atopos.
	We write $ \Gamma_{\X,\ast} $ or $ \Gamma(\X;-) $ for the \emph{global sections functor} $ \fromto{\X}{\Spc} $.
	We write $ \Gammaupperstar_{\X} \colon \fromto{\Spc}{\X} $ for the left adjoint to $ \Gamma_{\X,\ast} $, the \emph{constant sheaf functor}.
	Given a scheme $ X $, we also write $ \Gammaet(X;-) $ for $ \Gamma(X_{\et};-) $.
\end{notation}

\begin{definition}[(torsion lisse sheaf)]\label{def:Sigma-torsion_lisse_sheaf}
	Let $ \X $ be \atopos, $ F \in \X $, and $ \Sigma $ a set of prime numbers.
	\begin{enumerate}[label=\stlabel{def:Sigma-torsion_lisse_sheaf}, ref=\arabic*]
		\item We say that $ F $ is \defn{locally constant} if there exists a cover $\{U_{i}\}_{i \in I}$ of the terminal object of $ \X $, a corresponding family $ \{K_{i}\}_{i \in I} $ of spaces, and for each $ i \in I $, an equivalence 
		\begin{equation*}
			F \times U_{i}\simeq \Gammaupperstar_{\X}(K_{i}) \cross U_i
		\end{equation*}
	  	in $ \X_{/U_i} $.

		\item We say that $ F $ is \defn{\Sigmatorsion lisse} if $ F $ is locally constant and, in addition, the set $ I $ can be chosen to be finite and the spaces $K_{i}$ can be chosen to be \Sigmafinite.
		In the case that $ \Sigma $ is the set of all primes, we simply say that $ F $ is \defn{lisse}.
		We write $ \X^{\lisse} \subset \X $ for the full subcategory spanned by the lisse objects.
	\end{enumerate}
\end{definition}

\begin{definition}[(torsion étale sheaf)]\label{def:Sigma-torsion_etale_sheaf}
	Let $ X $ be a qcqs scheme, $ F \in X_{\et} $ an étale sheaf, and $ \Sigma $ a set of prime numbers.
	\begin{enumerate}[label=\stlabel{def:Sigma-torsion_etale_sheaf}, ref=\arabic*]
		\item We say that $ F $ is \defn{\Sigmatorsion constructible} if there exists a finite poset $P$ and a stratification $ \{X_p\}_{p \in P} $ of $ X $ by qcqs locally closed subschemes such that for each $ p \in P $, the sheaf $ \restrict{F}{X_p} $ is \Sigmatorsion lisse.
		In the case that $ \Sigma $ is the set of all primes, we simply say that $ F $ is \defn{constructible}.

		\item We say that $ F $ is \defn{\Sigmatorsion} if $ F $ is truncated and $ F $ can be written as the colimit of a filtered diagram of \Sigmatorsion constructible étale sheaves on $ X $.
		In the case that $ \Sigma $ is the set of all primes, we simply say that $ F $ is \defn{torsion}.
	\end{enumerate}
\end{definition}

\begin{remark}
	For a qcqs scheme $ X $, the subcategory of $ X_{\et} $ spanned by the \Sigmatorsion sheaves is closed under finite limits and truncations.
\end{remark}


\subsection{Gerbes in \texorpdfstring{$\infty$}{∞}-topoi}\label{subsec:gerbes_in_topoi}

In this subsection, we quickly recall the theory of gerbes in \atopos from \cite[\HTTsubsec{7.2.2}]{HTT}.

\begin{notation}
	Let $ \X $ be \atopos, $ F \in \X $, and $ n \geq 0 $ an integer.
	We write $ \uppi_n(F) \in (\X_{/F})_{\leq 0} $ for the $ n $-th homotopy object of $ F $, see \HTT{Definition}{6.5.1.1}.
	If $ n \geq 1 $, then $ \uppi_n(F) $ is naturally a group object of $ (\X_{/F})_{\leq 0} $, which is abelian if $ n \geq 2 $.

	Given a global section $ \gamma \colon \fromto{\ast}{F} $, we write $ \uppi_n(F,\gamma) \colonequals \gammaupperstar \uppi_n(F) $ for the pullback of $ \uppi_n(F) $ along $ \gamma $.
	Then for $ n \geq 1 $, the object $ \uppi_n(F,\gamma) $ is a group object of $ \X_{\leq 0} $, which is abelian if $ n \geq 2 $.
\end{notation}

\begin{recollection}[($ n $-gerbes)]
	Let $ \X $ be \atopos, $ F \in \X $, and $ n \geq 2 $ an integer. 
	Then $ F $ is called an \defn{$ n $-gerbe on $ \X $} if it is $ n $-truncated and $ (n-1) $-connected.%
	\footnote{We use the terminology for connectedness explained in \cites[\S3.3]{MR4186137}[\S4.1]{MR4334846}.}
	If $ F $ is an $ n $-gerbe, then the functor
	\begin{equation*}
		F \times (-) \colon \X \to \X_{/F}
	\end{equation*}
	restricts to an equivalence on $ (n-1) $-truncated objects \HTT{Lemma}{7.2.1.13}; in particular, on $ 0 $-truncated objects.
	In particular there is an unique abelian group object $ A \in \Abobj(\X_{\leq 0}) $ such that
	\begin{equation*}
		\uppi_n(F) \equivalent F \times A
	\end{equation*}
	in $ \X_{/F} $.
	In this case, we say that $ F $ is \defn{banded by $ A $}.
	Write $ p^F_! \colon \X_{/F} \to \X $ for the forgetful functor; the proof of \HTT{Lemma}{7.2.1.13} shows that there is an equivalence
	\begin{equation*}
		A \equivalent \trun_{\leq 0} p^F_!(\uppi_n(F)) \period
	\end{equation*}
\end{recollection}

\begin{recollection}[(Eilenberg--MacLane objects)]
	Let $ \X $ be \atopos and $ n \geq 1 $.
	A \defn{degree $ n $ Eilenberg--MacLane object} of $ \X $ is a pointed object $ \fromto{\pt}{F} $ that is $ n $-truncated and $ (n-1) $-connected.
	We write
	\begin{equation*}
		\EM_n(\X) \subset \X_{\ast}
	\end{equation*}
	for the full subcategory spanned by the degree $ n $ Eilenberg--MacLane objects.
	For $ n \geq 2 $, given a degree $ n $ Eilenberg--MacLane object $ \gamma \colon \ast \to F $, the $ n $-gerbe $ F $ is banded by the sheaf of homotopy groups $ \uppi_n(F,\gamma) \in \X_{\leq 0} $. 
	In particular we have an isomorphism
	\begin{equation*}
		\trun_{\leq 0} p^F_!(\uppi_n(F)) \cong \uppi_n(F,\gamma) \period
	\end{equation*}

	The functor 
	\begin{align*}
		\EM_1(\X) &\longrightarrow \Grpobj(\X_{\leq 0}) \\
		[\gamma \colon \pt \to F] &\longmapsto \uppi_1(F,\gamma)
	\end{align*}
	defines an equivalence between the \category of degree $ 1 $ Eilenberg--MacLane objects and the $ 1 $-category of group objects of $ \X_{\leq 0} $, see \HTT{Proposition}{7.2.2.12}.
	For $ n \geq 2 $, the functor
	\begin{align*}
		\EM_n(\X) &\longrightarrow \Abobj(\X_{\leq 0}) \\
		[\gamma \colon \pt \to F] &\longmapsto \uppi_n(F,\gamma)
	\end{align*}
	defines an equivalence between the \category of degree $ n $ Eilenberg--MacLane objects and the $ 1 $-category of abelian group objects of $ \X_{\leq 0} $.
	We write
	\begin{equation*}
		\Kup(-,1) \colon \equivto{\Grpobj(\X_{\leq 0})}{\EM_1(\X)} \andeq \Kup(-,n) \colon \equivto{\Abobj(\X_{\leq 0})}{\EM_n(\X)}
	\end{equation*}
	for the inverse equivalences.
\end{recollection}

\begin{recollection}[(cohomology in \atopos)]
	Let $ \X $ be \atopos, and let $ A $ be an abelian group object of $ \X_{\leq 0} $.
	For each integer $ n \geq 0 $, we write
	\begin{equation*}
		\Hup^n(\X;A) \colonequals \uppi_0 \Gamma(\X;\Kup(A,n))
	\end{equation*}
	for the \defn{$ n $-th cohomology group} of $ \X $ with coefficients in $ A $.
\end{recollection}

We recall the following classification result for gerbes in \atopos $ \X $. It is an immediate consquence of \HTT{Theorem}{7.2.2.26}.

\begin{theorem}\label{thm:classificationofgerbes}
	Let $ \X $ be \atopos, and let $ A $ be an abelian group object of $  \X_{\leq 0} $, and let $ n \geq 2 $ be an integer.
	Let $ \gamma \colon \ast \to \Kup(A,n+1) $ be a degree $ n+1 $ Eilenberg--MacLane object of $ \X $.
	Given an $ n $-gerbe $ F $ on $ \X $ banded by $ A $, there is a map $ \alpha_F \colon \ast \to \Kup(A,n+1) $ that is uniquely determined by $ F $ up to equivalence, and a pullback square
	\begin{equation*}
		\begin{tikzcd}
			F \arrow[r] \arrow[d] & \ast \arrow[d, "\gamma"] \\
			\ast \arrow[r, "\alpha_F"'] & \Kup(A,n+1) \period
		\end{tikzcd}
	\end{equation*} 
	In particular sending a map $ \alpha \colon \ast \to \Kup(A,n+1) $ to the above pullback defines a bijection
	\begin{equation*}
		\Hup^{n+1}(\X;A) \bijection \{\textup{$ n $-gerbes on $ \X $ banded by $ A $}\}/\kern-0.2em\equivalent \period
	\end{equation*}
\end{theorem}

\noindent As a direct consequence we obtain the following (see \HTT{Remark}{7.2.2.28}):

\begin{corollary}
	\label{cor:global_section_iff_vanishing}
	Let $ \X $ be \atopos, let $ A $ be an abelian group object of $  \X_{\leq 0} $, and let $ n \geq 2 $ be an integer.
	An $ n $-gerbe $ F $ banded by $ A $ admits a global section if and only if the corresponding cohomology class $ \alpha_F \in \Hup^{n+1}(\X;A) $ vanishes.
\end{corollary}

Later we need to use the fact that the banding of a \Sigmatorsion étale sheaf of spaces is a \Sigmatorsion étale sheaf of abelian groups.
For the proof, we need the following simple observation.

\begin{lemma}\label{lem:properties_of_cotruncation}
	Let $ \X $ be \atopos and $ n \geq 1 $ an integer.
	\begin{enumerate}[label=\stlabel{lem:properties_of_cotruncation}, ref=\arabic*]
		\item\label{lem:properties_of_cotruncation.1} The functor 
		\begin{align*}
			\trun_{\geq n} \colon \X_{\pt} &\to \X_{\pt} \\ 
			F &\mapsto \fib(F \to \trun_{\leq n-1}(F))
		\end{align*}
		preserves filtered colimits.

		\item\label{lem:properties_of_cotruncation.2} For all $ F \in \X_{\pt} $, the object $ \trun_{\geq n}(F) $ is $ (n-1) $-connected.

		\item\label{lem:properties_of_cotruncation.3} If $ F \in \X_{\pt} $ is $ n $-truncated, then $ \trun_{\geq n}(F) $ is a degree $ n $ Eilenberg--MacLane object.

		\item\label{lem:properties_of_cotruncation.4} If $ F \in \X_{\pt} $ is $ (n-1) $-connected, then the natural morphism $ \fromto{\trun_{\geq n}(F)}{F} $ is an equivalence.
	\end{enumerate}
\end{lemma}

\begin{proof}
	For \enumref{lem:properties_of_cotruncation}{1}, first note that the $ (n-1) $-truncation functor $ \trun_{\leq n-1} \colon \fromto{\X}{\X} $ preserves filtered colimits (\Cref{rec:n-truncation_preserves_filtered_colimits}), and that filtered colimits commute with finite limits in \atopos.
	Hence the claim follows from the fact that the forgetful functor $ \fromto{\X_{\pt}}{\X} $ creates limits and weakly contractible colimits.

	For \enumref{lem:properties_of_cotruncation}{2}, notice that the unit $ \fromto{F}{\trun_{\leq n}(F)} $ is $ (n-1) $-connected.
	Since $ (n-1) $-connected morphisms are stable under pullback \HTT{Proposition}{6.5.1.16}, the projection  $ \fromto{\trun_{\geq n}(F)}{\pt} $ is $ (n-1) $-connected, as desired.

	For \enumref{lem:properties_of_cotruncation}{3}, by definition we need to check that the pointed object $ \trun_{\geq n}(F) $ is $ (n-1) $-connected and $ n $-truncated.
	The $ (n-1) $-connectedness follows from \enumref{lem:properties_of_cotruncation}{2}.
	The $ n $-truncatedness of $ \trun_{\geq n}(F) $ follows from the fact that $ F $, $ \trun_{\leq n-1}(F) $, and $ \pt $ are $ n $-truncated and the full subcategory $ \X_{\leq n} \subset \X $ is closed under limits.

	For \enumref{lem:properties_of_cotruncation}{4}, note that since $ F $ is $ (n-1) $-connected, $ \trun_{\leq n-1}(F) \equivalent \pt $.
	Hence the claim follows from the fact that equivalences are stable under pullback.
\end{proof}

\begin{proposition}\label{prop:pi_n_of_torsion_gerbe_torsion}
	Let $ X $ be a qcqs scheme and $ \Sigma $ a set of prime numbers.
	Let $ F \in X_{\et} $ be a \Sigmatorsion sheaf of spaces on $ X $. 
	If $ F $ is an $ n $-gerbe, then the banding $ A_F \in \Abobj(X_{\et, \leq 0}) $ of $ F $ is a \Sigmatorsion étale sheaf of abelian groups on $ X $ (i.e., its stalks are \Sigmatorsion groups). 
\end{proposition}

\begin{proof}
	Since all functors involved in the construction of the banding $A_F \colonequals \trun_{\leq 0}p^F_!(\uppi_n F) $ are compatible with restriction along an étale map $ U \to X $, the claim is étale local on $ X $.
	Therefore we may assume that the gerbe $ F $ admits a global section $ \gamma \colon \ast \to F $, hence is a degree $ n $ Eilenberg--MacLane object.
	We can now write $ F $ as the filtered colimit of a diagram $ F_{\bullet} \colon \fromto{\Ical}{X_{\et}} $ of $ n $-truncated \Sigmatorsion constructible sheaves.
	Since the étale \topos $ X_{\et} $ is coherent \cite[Proposition 3.7.3]{arXiv:1807.03281}, the terminal object $ \ast \in X_{\et,\leq n} $ is compact \SAG{Propostion}{A.2.3.1}.
	Thus there is some $ i_0 \in \Ical $ such that $ \gamma $ factors through $ F_{i_0} $.
	Since the forgetful functor $ \fromto{\Ical_{i_0/}}{\Ical} $ is colimit-cofinal \cite[\HTTthm{Example}{5.4.5.9} \& \HTTthm{Lemma}{5.4.5.12}]{HTT}, replacing $ \Ical $ by $ \Ical_{i_0/} $ it follows that we may write the pointed object $ (F,\gamma) $ as a filtered colimit of \textit{pointed} objects $ (F_i,\gamma_i) $ such that each $ F_i $ is $ n $-truncated and \Sigmatorsion constructible.

	\Cref{lem:properties_of_cotruncation} shows that by replacing $ F_i $ by
	\begin{equation*}
		\trun_{\geq n}(F_i) = \fib( F_i \to \trun_{\leq n-1} F_i)
	\end{equation*}
	we may assume that all $ (F_i,\gamma_i) $ are also degree $ n $ Eilenberg--MacLane objects.
	Since the functor 
	\begin{equation*}
		\uppi_n(-) \colon \EM_n(X_{\et}) \to  X_{\et,\leq 0}
	\end{equation*}
	preserves filtered colimits, we may thus assume that $ F $ is \Sigmatorsion constructible and degree $ n $ Eilenberg--MacLane.
	Now let $ \fromto{x}{X} $ be a geometric point.
	Since homotopy groups are compatible with taking stalks, the group $ (A_F)_x = \uppi_n(F,\gamma)_x$ is isomorphic to $ \uppi_n(F_x,\gamma_x) $.
	But because $ F $ is \Sigmatorsion constructible, $ F_x $ is a \Sigmafinite space and therefore $ \uppi_n(F_x,\gamma_x)  $ is \Sigmafinite, as desired.
\end{proof}


\subsection{The étale homotopy type}\label{subsec:the_etale_homotopy_type}

In this paper, we make use of the description of the étale homotopy type of Artin--Mazur--Friedlander \cites[\S 9]{MR0245577}[\S 4]{MR676809} via Lurie's shape theory for \topoi.
In this subsection, we recall what we need of the theory.
We refer the reader to \cites[Chapters 4 \& 11]{arXiv:1807.03281}[\S2]{MR4367219}[\S2]{arXiv:1905.06243} for more background on shape theory and to \cite[\S 5]{MR3763287} for the relation to the classical definition of the étale homotopy type.

We begin by setting our notation for \proobjects and completions of prospaces.

\begin{notation}
	Given \acategory $ \Ccal $, we write $ \Pro(\Ccal) $ for the \category of \proobjects in $ \Ccal $ obtained by formally adjoining cofiltered limits to $ \Ccal $.
	The existence of $ \Pro(\Ccal) $ is a special case of (the dual of) \HTT{Proposition}{5.3.6.2}.
\end{notation}

\noindent We make extensive use of the following explicit presentation of the \category of \proobjects.

\begin{recollection}[{\cite[\SAGthm{Definition}{A.8.1.1} \& \SAGthm{Proposition}{A.8.1.6}]{SAG}}]\label{rec:pro-objects_as_left_exact_accessible_functors}
	Let $ \Ccal $ be an accessible \category with finite limits (e.g., $ \Ccal = \Spc $).
	Then there is a natural identification
	\begin{equation*}
		\Pro(\Ccal) \equivalent \Funlexacc(\Ccal,\Spc)^{\op}
	\end{equation*}
	with the opposite of the \category of left exact accessible functors $ \fromto{\Ccal}{\Spc} $.
\end{recollection}

\begin{remark}[{\HTT{Corollary}{5.4.3.6}}]
	Let $ \Ccal $ be a small \category.
	Then $ \Ccal $ is accessible if and only if $ \Ccal $ is idempotent complete.
	Moreover, if $ \Ccal $ is accessible, then given an accessible \category $ \Dcal $, every functor $ \fromto{\Ccal}{\Dcal} $ is accessible.
\end{remark}

\begin{nul}
	In particular, for every set of primes $ \Sigma $, the small \category $ \SpcSigma $ is accessible and every functor $ \fromto{\SpcSigma}{\Spc} $ is accessible.
\end{nul}

\begin{recollection}[(\Sigmacompletion)]
	Let $ \Sigma $ be a set of prime numbers.
	The inclusion functor $ \ProSpcSigma \subset \ProSpc $ admits a left adjoint
	\begin{equation*}
		(-)\Sigmacomp \colon \fromto{\ProSpc}{\ProSpcSigma}
	\end{equation*}
	called \defn{\Sigmacompletion}.
	If $ \Sigma $ is the set of all primes, we simply refer to \Sigmacompletion as \defn{profinite completion}.

	Under the identifications
	\begin{equation*}
		\ProSpc \equivalent \Funlexacc(\Spc,\Spc)^{\op} \andeq \ProSpcSigma \equivalent \Funlex(\SpcSigma,\Spc)^{\op} 
	\end{equation*}
	the functor $ (-)\Sigmacomp $ admits a very convenient description: it is given by pre-composition with the inclusion $ \SpcSigma \subset \Spc $.
\end{recollection}

Now we recall the basics of shape theory.

\begin{recollection}[(shape of \atopos)]\label{rec:shape}
	Write $ \RTop_{\infty} $ for the \category of \topoi and (right adjoints in) geometric morphisms.
	The \defn{shape} is a left adjoint functor
	\begin{equation*}
		\Shape \colon \fromto{\RTop_{\infty}}{\ProSpc}
	\end{equation*}
	that admits the following explicit description.
	\begin{enumerate}[label=\stlabel{rec:shape}, ref=\arabic*]
		\item Given \atopos $ \X $, the shape $ \Shape(\X) $ is the left exact accessible functor $ \fromto{\Spc}{\Spc} $ given by the composite
		\begin{equation*}
			\Gamma_{\X,\ast}\Gammaupperstar_{\X} \colon \fromto{\Spc}{\Spc} \period 
		\end{equation*}
		That is, for each space $ K $, the value of $ \Shape(\X) $ on $ K $ is the global sections of the constant object of $ \X $ with value $ K $.

		\item Given a geometric morphism $ \flowerstar \colon \fromto{\X}{\Y} $ with unit $ \unit \colon \fromto{\id{\Y}}{\flowerstar \fupperstar} $, the induced morphism of prospaces $ \fromto{\Shape(\X)}{\Shape(\Y)} $ corresponds to the morphism
		\begin{equation*}
			\Gamma_{\Y,\ast} \unit \Gammaupperstar_{\Y} \colon \Gamma_{\Y,\ast} \Gammaupperstar_{\Y} \longrightarrow \Gamma_{\Y,\ast} \flowerstar \fupperstar \Gammaupperstar_{\Y} \equivalent \Gamma_{\X,\ast} \Gammaupperstar_{\X}
		\end{equation*}
		in $ \ProSpc^{\op} \subset \Fun(\Space,\Space) $.
	\end{enumerate}
	We refer the reader to \cites[\HTTsubsec{7.1.6}]{HTT}[\S 2]{MR3763287} for more details.
\end{recollection}

\begin{notation}
	Given \atopos $ \X $, we write $ \Shapeprofin(\X) $ for the profinite completion of $ \Shape(\X) $. 
	We call $ \Shapeprofin(\X) $ the \defn{profinite shape} of $ \X $.
\end{notation}

\begin{notation}
	Given a scheme $ X $, we write $ \Piet(X) \colonequals \Shape(X_{\et}) $ for the shape of the étale \topos of $ X $.
	We call $ \Piet(X) $ the \defn{étale homotopy type} of $ X $.
	We write $ \Pietprofin(X) $ for the profinite shape of $ X_{\et} $ and refer to $ \Pietprofin(X) $ as the \defn{profinite étale homotopy type} of $ X $.
\end{notation}

We make frequent use of the following reformulation of what it means for a geometric morphism to induce an equivalence on shapes.

\begin{observation}
	Let $ \flowerstar \colon \fromto{\X}{\Y} $ be a geometric morphism of \topoi and let $ \Sigma $ be a set of prime numbers.
	Then the induced map $ \fromto{\Shape(\X)}{\Shape(\Y)} $ is an equivalence if and only if for each space $ K $, the induced map on global sections $ \fromto{\Gamma(\Y;K)}{\Gamma(\X;K)} $ is an equivalence.
	Similarly, the induced map 
	\begin{equation*}
		\fromto{\Shape(\X)\Sigmacomp}{\Shape(\Y)\Sigmacomp}
	\end{equation*}
	on \Sigmacomplete shapes is an equivalence if and only if for each \Sigmafinite space $ K $, the induced map on global sections $ \fromto{\Gamma(\Y;K)}{\Gamma(\X;K)} $ is an equivalence.
\end{observation}

One of the most important results about the profinite shape of \atopos is that it is characterized by the fact that it classifies lisse objects.
 
\begin{notation}
	Let $ \Ccal $ be \acategory.
	We write
	\begin{equation*}
		\Fun(-,\Ccal) \colon \fromto{\ProCat^{\op}}{\Catinfty}
	\end{equation*}
	for the unique functor that extends $ \Fun(-,\Ccal) \colon \fromto{\Catinfty^{\op}}{\Catinfty} $ and transforms cofiltered limits in $ \ProCat $ to filtered colimits in $ \Catinfty $.
\end{notation}

\begin{theorem}[{(monodromy for lisse objects \cite[Proposition 4.4.18]{arXiv:1807.03281})}]
	Let $ \X $ be \atopos.
	Then there is a natural equivalence of \categories
	\begin{equation*}
		\X^{\lisse} \equivalent \Fun(\Shapeprofin(\X),\Spcfin) \period
	\end{equation*} 
\end{theorem}

\begin{theorem}[{\SAG{Corollary}{E.2.3.3}}]\label{thm:equivalences_on_lisse_sheaves_and_profinite_shapes}
	Let $ \flowerstar \colon \fromto{\X}{\Y} $ be a geometric morphism of \topoi.
	The following are equivalent:
	\begin{enumerate}[label=\stlabel{thm:equivalences_on_lisse_sheaves_and_profinite_shapes}]
		\item The pullback functor $ \fupperstar $ restricts to an equivalence $ \equivto{\Y^{\lisse}}{\X^{\lisse}} $. 

		\item The induced map of profinite spaces $ \fromto{\Shapeprofin(\X)}{\Shapeprofin(\Y)} $ is an equivalence.
	\end{enumerate}
\end{theorem}


\subsection{Exchange transformations}\label{subsec:exchange_transformations}

We now recall the key compatibility of exchange transformations that we need to use in our reduction to strictly henselian local rings in the proofs of the nonabelian basechange theorems.

\begin{definition}
	Let 
	\begin{equation}\label{square:oriented}
		\begin{tikzcd}
			\Wcal \arrow[r, "\fbarlowerstar"] \arrow[d, "\gbarlowerstar"'] & \Ycal \arrow[d, "\glowerstar"] \arrow[dl, phantom, "\scriptstyle \sigma" below right, "\Longleftarrow" sloped] \\ 
			\Xcal \arrow[r, "\flowerstar"'] & \Zcal
		\end{tikzcd}
	\end{equation}
	be a square of \categories and functors commuting up to a natural transformation $ \sigma $.
	Assume that the functors $ \flowerstar $ and $ \fbarlowerstar $ admit left adjoints $ \fupperstar $ and $ \fbarupperstar $, respectively.
	Write $ \counit_f \colon \fromto{\fupperstar \flowerstar}{\id{\Xcal}}$ for the counit and $ \unit_{\fbar} \colon \fromto{\id{\Ycal}}{\fbarlowerstar\fbarupperstar}$ for the unit.
	The \defn{exchange transformation} associated to the oriented square \eqref{square:oriented} is the composite natural transformation
	\begin{equation*}
		\begin{tikzcd}[sep=3em]
			\Ex \colon \fupperstar \glowerstar \arrow[r, "\fupperstar \glowerstar\unit_{\fbar}"] & \fupperstar \glowerstar\fbarlowerstar\fbarupperstar \arrow[r, "\fupperstar \sigma \fbarupperstar"] & \fupperstar \flowerstar\gbarlowerstar\fbarupperstar \arrow[r, "\counit_f \gbarlowerstar\fbarupperstar"] & \gbarlowerstar\fbarupperstar \period 
		\end{tikzcd}
	\end{equation*}	
	Note that if $ \sigma $ is an equivalence (so that \eqref{square:oriented} commutes), then the middle morphism in the definition of $ \Ex $ is an equivalence.
\end{definition}

\noindent The following is immediate from the functoriality of the `mate correspondence'; see \cites[Theorem B \& Corollary F]{arXiv:2011.08808}[\S2.2]{MR0357542}[Theorem B.3.6]{MR4354541}.

\begin{proposition}\label{prop:compatibility_of_exchange_transformations}
    Let 
    \begin{equation*}\label{cube:general_categories}
        \begin{tikzcd}[column sep={8ex,between origins}, row sep={8ex,between origins}]
            \Wcal' \arrow[rr, "\pbarlowerstar"] \arrow[dd, "\qbarlowerstar"']  \arrow[dr, "\wlowerstar" description] & & \Ycal' \arrow[dd, "\qlowerstar" near end]  \arrow[dr, "\ylowerstar"] \\
            & \Wcal \arrow[rr, crossing over, "\fbarlowerstar" near start] & & \Ycal \arrow[dd, "\glowerstar"]  \\
            \Xcal' \arrow[rr, "\plowerstar"' near end] \arrow[dr, "\xlowerstar"'] & & \Zcal' \arrow[dr, "\zlowerstar" description] \\
            & \Xcal \arrow[rr, "\flowerstar"'] \arrow[from=uu, crossing over, "\gbarlowerstar"' near start] & & \Zcal \comma & 
        \end{tikzcd}
    \end{equation*}
    be a commutative cube of \categories and right adjoint functors.
    Then the square
    \begin{equation*}
        \begin{tikzcd}[sep=3em]
            \xupperstar \fupperstar \glowerstar \arrow[rr, "\xupperstar\Ex"] \arrow[d, "\wr"'{xshift=0.2ex}] & & \xupperstar \gbarlowerstar \fbarupperstar \arrow[d, "\Ex \fbarupperstar"] \\ 
            \pupperstar \zupperstar \glowerstar \arrow[d, "\pupperstar\Ex"'] & & \qbarlowerstar \wupperstar \fbarupperstar \arrow[d, "\wr"{xshift=-0.2ex}] \\ 
            \pupperstar \qlowerstar \yupperstar \arrow[rr, "\Ex\yupperstar"']  & & \qbarlowerstar \pbarupperstar \yupperstar 
        \end{tikzcd}
    \end{equation*}
    canonically commutes.
    Here the indicated equivalences are natural identifications of adjoints.
\end{proposition}

%% file: content/classical_basechange_to_nonabelian_basechange.tex

\section{From classical basechange to nonabelian basechange}\label{sec:reduction_to_local_rings}

The goal of this section is to prove nonabelian refinements of: the smooth basechange theorem (\Cref{cor:nonabelian_smooth_basechange}), the proper basechange theorem (\Cref{cor:nonabelian_proper_basechange}), the Gabber--Huber affine analogue of proper basechange (\Cref{cor:affine_analogue_of_proper_basechange}), and the Fujiwara--Gabber rigidity theorem (\Cref{cor:Fujiwara-Gabber_rigidity}).

In \cref{subsec:generalities_on_strictly_henselian_local_rings} we recall how to describe stalks of étale sheaves on strictly henselian local rings in terms of global sections.
In \cref{subsec:stalk_of_the_exchange_transformation}, we give a simple description of the stalk of an exchange transformation (\Cref{lem:compatibility_with_local_exchange_transformations}).
Using this description, \cref{subsec:reduction_to_the_local_case} proves the key technical results that let us reduce proving nonabelian basechange theorems to classical basechange results (\Cref{prop:comparing_global_sections,cor:global_sections_over_local_rings,cor:reduction_to_1-truncated_local}).
\Cref{subsec:nonabelian_basechange_theorems} deduces nonabelian refinements of all of the basechange theorems mentioned in the previous paragraph.
In \cref{subsec:invariance_under_specialization}, we apply the nonabelian smooth and proper basechange theorems to show that, after completion away from the residue characteristics, the étale homotopy types of the geometric fibers of a smooth proper morphism of schemes are invariant under specialization (see \Cref{prop:invariance_under_specalization}).


\subsection{Generalities on strictly henselian local rings}\label{subsec:generalities_on_strictly_henselian_local_rings}

We start with some general facts about strictly henselian local rings, specifically that global sections can be computed as the stalk at the closed point.

\begin{notation}\label{ntn:local_rings}
    Let $ A $ and $ B $ be strictly henselian local rings, and $ \phi \colon \fromto{B}{A} $ a local ring homomorphism.
    Write $ f \colon \fromto{\Spec(A)}{\Spec(B)} $ for the induced map on spectra.
    Write $ \fromto{x}{\Spec(A)} $ and $ \fromto{z}{\Spec(B)} $ for the geometric points specified by the residue fields of $ A $ and $ B $, respectively.
\end{notation}

\begin{proposition}[(stalks via global sections)]\label{prop:stalks_via_global_sections_for_strictly_henselian_local_rings}
    In the setting of \Cref{ntn:local_rings}: 
    \begin{enumerate}[label=\stlabel{prop:stalks_via_global_sections_for_strictly_henselian_local_rings}, ref=\arabic*]
        \item\label{prop:stalks_via_global_sections_for_strictly_henselian_local_rings.1} There is a natural equivalence
        \begin{equation*}
            \xupperstar \equivalent \Gammaet(\Spec(A);-) 
        \end{equation*} 
        of functors $ \fromto{\Spec(A)_{\et}}{\Spc} $.
        Hence $ \xlowerstar \colon \fromto{\Spc}{\Spec(A)_{\et}} $ is right adjoint to the global sections functor.

        \item\label{prop:stalks_via_global_sections_for_strictly_henselian_local_rings.2} There are natural equivalences
        \begin{equation*}
            \Gammaet(\Spec(A);\fupperstar(-)) \equivalent 
            \xupperstar \fupperstar \equivalent \zupperstar \equivalent \Gammaet(\Spec(B);-)
        \end{equation*} 
        of functors $ \fromto{\Spec(B)_{\et}}{\Spc} $.
    \end{enumerate}
\end{proposition}

\begin{proof}
    The proof of \enumref{prop:stalks_via_global_sections_for_strictly_henselian_local_rings}{1} in the setting of $ 1 $-topoi is given in \cite[Exposé VIII, Proposition 4.6]{MR50:7131}.
    The proof given there works \textit{verbatim} in the setting of \topoi. 
    For \enumref{prop:stalks_via_global_sections_for_strictly_henselian_local_rings}{2}, note that since $ \phi $ is a local ring homomorphism, we have a commutative square of schemes
    \begin{equation*}
        \begin{tikzcd}
            x \arrow[r] \arrow[d, hooked] & z \arrow[d, hooked] \\ 
            \Spec(A) \arrow[r, "f"'] & \Spec(B) \period
        \end{tikzcd}
    \end{equation*}
    Passing to étale \topoi, we obtain a commutative square of \topoi
     \begin{equation}\label{sq:etale_topoi_of_local_rings}
        \begin{tikzcd}
            x_{\et} \arrow[r] \arrow[d, hooked] & z_{\et} \arrow[d, hooked] \\ 
            \Spec(A)_{\et} \arrow[r, "\flowerstar"'] & \Spec(B)_{\et} \period
        \end{tikzcd}
    \end{equation}
    Since $ x $ and $ z $ are geometric points, $ x_{\et} \equivalent \Spc $ and $ z_{\et} \equivalent \Spc $.
    Since $ \Spc $ is the terminal \topos, we deduce that the top horizontal geometric morphism in \eqref{sq:etale_topoi_of_local_rings} is the identity.
    This shows that $ \xlowerstar \flowerstar \equivalent \zlowerstar $; equivalently, $ \fupperstar \xupperstar \equivalent \zupperstar $.
    Applying \enumref{prop:stalks_via_global_sections_for_strictly_henselian_local_rings}{1} to both of the strictly henselian local rings $ A $ and $ B $, we deduce that 
    \begin{align*}
        \Gammaet(\Spec(A);\fupperstar(-)) &\equivalent \fupperstar \xupperstar \\
        &\equivalent \zupperstar \equivalent \Gammaet(\Spec(B);-) \period \qedhere
    \end{align*}
\end{proof}

In this setting, the constant sheaf functor is also fully faithful:

\begin{lemma}
    Let $ \X $ be \atopos.
    If the global sections functor $ \Gammalowerstar \colon \fromto{\X}{\Spc} $ admits a right adjoint \smash{$ \Gammauppersharp \colon \fromto{\Spc}{\X} $}, then both \smash{$ \Gammauppersharp $} and \smash{$ \Gammaupperstar $} are fully faithful.
\end{lemma}

\begin{proof}
    To see that $ \Gammauppersharp $ is fully faithful, first note that $ \Gammauppersharp $ is a geometric morphism.
    Since $ \Spc $ is the terminal \topos \HTT{Proposition}{6.3.4.1}, the composite geometric morphism
    \begin{equation*}
        \Gammalowerstar \Gammauppersharp \colon \fromto{\Spc}{\Spc} 
    \end{equation*}
    is equivalent to $ \id{\Spc} $.
    Since $ \Gammauppersharp $ is right adjoint to $ \Gammalowerstar $, this implies that $ \Gammauppersharp $ is fully faithful \cite[Lemma 3.3.1]{arXiv:2007.13089}.
    
    The claim that $ \Gammaupperstar $ is fully faithful now follows from the general fact that given a triple of adjoints $ \fupperstar \leftadjoint \flowerstar \leftadjoint \fuppersharp $, the functor $ \fupperstar $ is fully faithful if and only if $ \fuppersharp $ is fully faithful.
\end{proof}

\begin{example}\label{obs:the_constant_sheaf_functor_of_a_local_topos_is_fully_faithful}
    For a strictly henselian local ring $ A $, the constant sheaf functor $ \fromto{\Spc}{\Spec(A)_{\et}} $ is fully faithful.
\end{example}

\begin{lemma}\label{lem:local_exchange_in_terms_of_global_sections}
    Keep \Cref{ntn:local_rings}, and let 
    \begin{equation*}
        \begin{tikzcd}
            W \arrow[r, "\fbar"] \arrow[d, "\gbar"'] & Y \arrow[d, "g"] \\ 
            \Spec(A) \arrow[r, "f"'] & \Spec(B) 
        \end{tikzcd}
    \end{equation*}
    be a commutative square of schemes.
    Then there is a natural commutative square
    \begin{equation*}
        \begin{tikzcd}[sep=2.5em]
            \xupperstar \fupperstar \glowerstar \arrow[r, "\xupperstar\Ex"] \arrow[d, "\wr"'{xshift=0.25ex}] & \xupperstar \gbarlowerstar \fbarupperstar \arrow[d, "\wr"{xshift=-0.25ex}] \\
            \Gammaet(Y;-) \arrow[r] & \Gammaet(W;\fbarupperstar(-))
        \end{tikzcd}
    \end{equation*}
    of functors $ \fromto{Y_{\et}}{\Spc} $.
    Here the vertical maps are equivalences and the bottom horizontal map is induced by the unit $ \fromto{\id{}}{\fbarlowerstar\fbarupperstar} $.
\end{lemma}

\begin{proof}
    To simplify notation, write $ X = \Spec(A) $ and $ Z = \Spec(B) $.
    Since $ \xupperstar \equivalent \Gammaet(X;-) $, we equivalently need to show that there is a commutative square with $ \xupperstar $ replaced by $ \Gammaet(X;-) $.
    To see this, consider the commutative diagram of \topoi
    \begin{equation}\label{rectang:square_and_global_sections}
        \begin{tikzcd}[sep=2.5em]
            W_{\et} \arrow[r, "\fbarlowerstar"] \arrow[d, "\gbarlowerstar"'] & Y_{\et} \arrow[d, "\glowerstar"] \\ 
            X_{\et} \arrow[r, "\flowerstar"'] \arrow[d, "\Gammaet(X;-)"'] & Z_{\et} \arrow[d, "\Gammaet(Z;-)"] \\ 
            \Spc \arrow[r, equals] & \Spc \period
        \end{tikzcd}
    \end{equation}
    By the functoriality of exchange transformations, there is a commutative diagram
    \begin{equation*}
        \begin{tikzcd}[row sep=2.5em, column sep=5.5em]
            \Gammaet(Z;\glowerstar(-)) \arrow[r, "\Ex\glowerstar"] \arrow[d, "\wr"'{xshift=0.25ex}] & \Gammaet(X;\fupperstar\glowerstar(-)) \arrow[r, "\Gammaet(X;-)\Ex"] & \Gammaet(X;\gbarlowerstar\fbarupperstar(-)) \arrow[d, "\wr"{xshift=-0.25ex}] \\ 
            \Gammaet(Y;-) \arrow[rr, "\Ex"'] & & \Gammaet(W;\fbarupperstar(-)) \period
        \end{tikzcd}
    \end{equation*}
    Here the vertical equivalences are identifications of adjoints, the top left-hand horizontal morphism is induced by the exchange transformation associated to the bottom square of \eqref{rectang:square_and_global_sections}, the top right-hand horizontal morphism is induced by the exchange transformation associated to the top square of \eqref{rectang:square_and_global_sections}, and the bottom horizontal morphism is the exchange transformation associated to the large outer rectangle in \eqref{rectang:square_and_global_sections}.

    To complete the proof, note that by \Cref{prop:stalks_via_global_sections_for_strictly_henselian_local_rings} and the uniqueness of the global sections functor, the exchange transformation $ \fromto{\Gammaet(Z;-)}{\Gammaet(X;\fupperstar(-))} $ is an equivalence.
    Moreover, since the bottom horizontal functor in the diagram \eqref{rectang:square_and_global_sections} is the identity, unpacking the definition shows that the exchange transformation associated to the large outer rectangle in \eqref{rectang:square_and_global_sections} is given by applying $ \Gamma(Y;-) $ to the unit $ \fromto{\id{}}{\fbarlowerstar\fbarupperstar} $.
\end{proof}


\subsection{The stalk of the exchange transformation}\label{subsec:stalk_of_the_exchange_transformation}

Let 
\begin{equation}\label{sq:general_square}
    \begin{tikzcd}
        W \arrow[r, "\fbar"] \arrow[d, "\gbar"'] & Y \arrow[d, "g"] \\ 
        X \arrow[r, "f"'] & Z 
    \end{tikzcd}
\end{equation}
be a commutative square of qcqs schemes (which we fix throughout this subsection).
Given a geometric point $ \fromto{x}{X} $, the goal of this subsection is to use \Cref{lem:local_exchange_in_terms_of_global_sections} to express the stalk of the exchange transformation $ \fromto{\fupperstar\glowerstar}{\gbarlowerstar\fbarupperstar} $ at $ x $ in terms of global sections.
To explain this, we fix the following notation.

\begin{notation}[(strict localizations)]\label{ntn:strict_localization}
    Let $ S $ be a qcqs scheme and $ s \to S $ a geometric point.
    Write
    \begin{equation*}
        \Loc{S}{s} \colonequals \Spec(\Osh_{S,s}) 
    \end{equation*}
    for the \defn{strict localization} of $ S $ at $ s $.
    We write $ \el_{s} \colon \fromto{\Loc{S}{s}}{S} $ for the projection.
    Given a morphism of qcqs schemes $ X \to S $, we write $ \Xs $ and $ \Loc{X}{s} $ for the pullbacks of schemes
    \begin{equation*}
        \begin{tikzcd}[sep=2.25em]
           \Xs \arrow[dr, phantom, very near start, "\lrcorner", xshift=-0.25em, yshift=0.12em] \arrow[r] \arrow[d] & \Loc{X}{s}  \arrow[dr, phantom, very near start, "\lrcorner", xshift=-0.25em, yshift=0.25em] \arrow[r, "\elbar_{s}"] \arrow[d, "f_{(s)}"'] & X \arrow[d, "f"] \\ 
           s \arrow[r] & \Loc{S}{s} \arrow[r, "\el_{s}"'] & S \period
        \end{tikzcd}
      \end{equation*}
\end{notation}

The key fact we need is that the right-hand square satisfies basechange for truncated sheaves:

\begin{proposition}\label{prop:global-to-local-base-change}
    Keep \Cref{ntn:strict_localization}.
    Then for each truncated étale sheaf $ F \in X_{\et, <\infty} $, the exchange transformation
    \begin{equation*}
        \elupperstar_s \flowerstar(F) \to f_{(s),\ast} \elbarupperstar_s(F)
    \end{equation*}
    is an equivalence.
\end{proposition}

\begin{proof}
    Combine \cites[Example 6.7.4 \& Proposition 7.5.1]{arXiv:1807.03281} with \cite[Proposition 2.3]{arXiv:2209.03476}.
\end{proof}

We make use of the following notation throughout the rest of the section:

\begin{notation}
    Consider a commutative square \eqref{sq:general_square} of qcqs schemes.
    Let $ \fromto{x}{X} $ be a geometric point with image $ \fromto{z}{Z} $ in $ Z $.
    We denote the natural morphisms between these schemes induced by the functionality of pullbacks as indicated in the following commutative cube: 
    \begin{equation}\label{cube:reduction_to_local_rings}
      \begin{tikzcd}[column sep={8ex,between origins}, row sep={8ex,between origins}]
            \Loc{W}{x}  \arrow[rr, "\pbar"] \arrow[dd, "\qbar"']  \arrow[dr, "\elbar_{x}" description] & & \Loc{Y}{z} \arrow[dd, "q" near end]  \arrow[dr, "\elbar_{z}"] \\
            & W \arrow[rr, crossing over, "\fbar" near start] & & Y \arrow[dd, "g"]  \\
            \Loc{X}{x} \arrow[rr, "p"' near end] \arrow[dr, "\el_{x}"'] & & \Loc{Z}{z} \arrow[dr, "\el_{z}" description] \\
            & X \arrow[rr, "f"'] \arrow[from=uu, crossing over, "\gbar"' near start] & & Z \period
      \end{tikzcd}
    \end{equation}
    By definition, the side vertical faces of \eqref{cube:reduction_to_local_rings} are pullback squares.
    Hence if the front vertical face is a pullback square, then the back vertical face is also a pullback square.
\end{notation}

We are ready to rewrite the stalk of the exchange transformation in terms of global sections:

\begin{lemma}\label{lem:compatibility_with_local_exchange_transformations}
    Consider a commutative square \eqref{sq:general_square} of qcqs schemes, and let $ \fromto{x}{X} $ be a geometric point with image $ \fromto{z}{Z} $.
    Then:
    \begin{enumerate}[label=\stlabel{lem:compatibility_with_local_exchange_transformations}, ref=\arabic*]
        \item\label{lem:compatibility_with_local_exchange_transformations.1} There is a commutative square 
        \begin{equation*}
            \begin{tikzcd}[column sep=4em, row sep=2.5em]
                \elupperstar_x \fupperstar \glowerstar \arrow[r, "\elupperstar_x \Ex"] \arrow[d] & \elupperstar_x \gbarlowerstar \fbarupperstar \arrow[d] \\
                \pupperstar \qlowerstar \elbarupperstar_z \arrow[r, "\Ex \elbarupperstar_z"'] & \qbarlowerstar \pbarupperstar \elbarupperstar_z
            \end{tikzcd}
        \end{equation*}
        of functors $ \fromto{Y_{\et}}{\Locet{X}{x}} $.
        Moreover, the vertical natural transformations are equivalences when evaluated on truncated étale sheaves.

        \item\label{lem:compatibility_with_local_exchange_transformations.2} There is a commutative square 
        \begin{equation*}
            \begin{tikzcd}[column sep=4em, row sep=2.5em]
                \xupperstar \fupperstar \glowerstar \arrow[r, "\xupperstar \Ex"] \arrow[d] & \xupperstar \gbarlowerstar \fbarupperstar \arrow[d] \\
                \Gammaet(\Loc{Y}{z};\elbarupperstar_z(-)) \arrow[r] & \Gammaet(\Loc{W}{z}; \pbarupperstar \elbarupperstar_z(-))
            \end{tikzcd}
        \end{equation*}
        of functors $ \fromto{Y_{\et}}{x_{\et} \equivalent \Spc } $.
        Moreover, the vertical natural transformations are equivalences when evaluated on truncated étale sheaves.
    \end{enumerate}
\end{lemma}

\begin{proof}
    Claim \enumref{lem:compatibility_with_local_exchange_transformations}{1} is an immediate consequence of \Cref{prop:global-to-local-base-change} combined with the functoriality of exchange transformations (\Cref{prop:compatibility_of_exchange_transformations}) applied to the diagram of étale \topoi associated to the cube of schemes \eqref{cube:reduction_to_local_rings}.
    By taking stalks at the geometric point $ \fromto{x}{\Loc{X}{x}} $, claim \enumref{lem:compatibility_with_local_exchange_transformations}{2} follows from \enumref{lem:compatibility_with_local_exchange_transformations}{1} combined with \Cref{lem:local_exchange_in_terms_of_global_sections}.
\end{proof}

\noindent The following gives a reformulation of when the exchange transformation is an equivalence.

\begin{corollary}\label{cor:when_is_the_exchange_transformation_an_equivalence}
    Consider a commutative square \eqref{sq:general_square} of qcqs schemes, and let $ F \in Y_{\et,<\infty} $ be a truncated étale sheaf on $ Y $.
    The following are equivalent: 
    \begin{enumerate}[label=\stlabel{cor:when_is_the_exchange_transformation_an_equivalence}, ref=\arabic*]
        \item\label{cor:when_is_the_exchange_transformation_an_equivalence.1} The exchange transformation $  \fromto{\fupperstar\glowerstar(F)}{\gbarlowerstar\fbarupperstar(F)} $ is an equivalence.

        \item\label{cor:when_is_the_exchange_transformation_an_equivalence.2} For each geometric point $ \fromto{x}{X} $, the stalk of the exchange transformation
        \begin{equation*}
            \fromto{\xupperstar\fupperstar\glowerstar(F)}{\xupperstar\gbarlowerstar\fbarupperstar(F)}
        \end{equation*}
        is an equivalence.

        \item\label{cor:when_is_the_exchange_transformation_an_equivalence.3} For each geometric point $ \fromto{x}{X} $ with image $ \fromto{z}{Z} $, the natural map
        \begin{equation*}
            \fromto{\Gammaet(\Loc{Y}{z};\elbarupperstar_z F)}{\Gammaet(\Loc{W}{z}; \pbarupperstar \elbarupperstar_z F)}
        \end{equation*}
        is an equivalence.
    \end{enumerate}
\end{corollary}

\begin{proof}
    Immediate from \Cref{lem:compatibility_with_local_exchange_transformations} and the fact that equivalences of truncated étale sheaves of spaces on a qcqs scheme can be checked on stalks \cites[Propositions \SAGthmlink{2.3.4.2} \& \SAGthmlink{A.4.0.5}]{SAG}.
\end{proof}


\subsection{Reduction to the local case}\label{subsec:reduction_to_the_local_case}

In light of \Cref{cor:when_is_the_exchange_transformation_an_equivalence}, the following topos-theoretic proposition provides a criterion for using basechange for sheaves of $ 1 $-groupoids and cohomological basechange for sheaves of abelian groups to deduce that the stalk of the exchange transformation is an equivalence.
We are most interested in the case of a pushforward on étale \topoi induced by a morphisms of schemes.

\begin{proposition}\label{prop:comparing_global_sections}
    Let $ \pilowerstar \colon \fromto{\Ubf}{\V} $ be a geometric morphism of \topoi and let $ \V' \subset \V_{<\infty} $ be a full subcategory closed under finite limits and truncations.
    Consider the following statements:
    \begin{enumerate}[label=\stlabel{prop:comparing_global_sections}, ref=\arabic*]
        \item\label{prop:comparing_global_sections.1} For each $ 1 $-truncated object $ F \in \V' $, the natural map $ \fromto{\Gamma(\V;F)}{\Gamma(\Ubf;\piupperstar F)} $ is an equivalence.

        \item\label{prop:comparing_global_sections.2} For each integer $ n \geq 2 $ and degree $ n $ Eilenberg--MacLane object $ G \in \V' $, the natural map $ \fromto{\Gamma(\V;G)}{\Gamma(\Ubf;\piupperstar G)} $ is an equivalence.

        \item\label{prop:comparing_global_sections.3} For each integer $ n \geq 2 $ and $ n $-gerbe $ G \in \V' $, the banding of $ G $ is in $ \V' $.

        \item\label{prop:comparing_global_sections.4} For each integer $ n \geq 2 $ and $ n $-gerbe $ G \in \V' $, the natural map $ \fromto{\Gamma(\V;G)}{\Gamma(\Ubf;\piupperstar G)} $ is an equivalence.

        \item\label{prop:comparing_global_sections.5} For each $ F \in \V' $, the natural map $ \fromto{\Gamma(\V;F)}{\Gamma(\Ubf;\piupperstar F)} $ is an equivalence.
    \end{enumerate}

    Then \enumref{prop:comparing_global_sections}{2} and \enumref{prop:comparing_global_sections}{3} together imply \enumref{prop:comparing_global_sections}{4}.
    Also \enumref{prop:comparing_global_sections}{1} and \enumref{prop:comparing_global_sections}{4} together imply \enumref{prop:comparing_global_sections}{5}.
    Hence \enumref{prop:comparing_global_sections}{1}, \enumref{prop:comparing_global_sections}{2}, and \enumref{prop:comparing_global_sections}{3} together imply \enumref{prop:comparing_global_sections}{5}. 
\end{proposition}

\begin{remark}\label{prop:comparing_global_sections.2_via_cohomology}
    Condition \enumref{prop:comparing_global_sections}{2} is equivalent to the condition that for each abelian group object $ A $ of $ (\V')_{\leq 0} $ and integer $ i \geq 0 $, the induced map on cohomology groups
    \begin{equation*}
        \fromto{\Hup^i(\V;A)}{\Hup^i(\Ubf;\piupperstar A)}
    \end{equation*}
    is an isomorphism.
\end{remark}

\begin{remark}
    \Cref{prop:comparing_global_sections} is surely known to experts.
    For example, in more specific situations, the argument we give essentially appears in \cite[Proposition 5.11]{MR4367219} and \cite[\S4]{MR4493612}.
    Since we could not find a reference for the result stated in its full generality, we have decided to record it for posterity.
\end{remark}

\begin{proof}[Proof that \enumref{prop:comparing_global_sections}{2} $ + $ \enumref{prop:comparing_global_sections}{3} $ \Rightarrow $ \enumref{prop:comparing_global_sections}{4}]
    Let $ n \geq 2 $ and let $ G \in \V' $ be an $ n $-gerbe.
    By assumption \enumref{prop:comparing_global_sections}{2}, all that remains to be shown is that if $ \Gamma(\V;G) = \emptyset $, then $ \Gamma(\Ubf;\piupperstar G) = \emptyset $.

    Write $ A \in \V $ for the banding of $ G $; by assumption \enumref{prop:comparing_global_sections}{3}, we have that $ A \in \V' $.
    Let
    \begin{equation*}
        \alpha_G \in \Hup^{n+1}(\V;A) \andeq \alpha_{\piupperstar G} \in \Hup^{n+1}(\Ubf;\piupperstar A) 
    \end{equation*}
    denote the cohomology classes corresponding to $ G $ and $ \piupperstar G $ via \Cref{thm:classificationofgerbes}. 
    Since $ \Gamma(\V;G) $ is empty, \Cref{cor:global_section_iff_vanishing} implies that $ \alpha_G $ is nonzero.
    Again by \Cref{cor:global_section_iff_vanishing}, the claim that $ \Gamma(\Ubf;\piupperstar G) = \emptyset $ is equivalent to the claim that the class $ \alpha_{\piupperstar G} $ is nonzero.
    To see that $ \alpha_{\piupperstar G} \neq 0 $, by applying assumption \enumref{prop:comparing_global_sections}{2} to the $ (n+1) $-gerbe $ \Kup(A,n+1) $, we see that the natural pullback map of abelian groups
    \begin{equation}\label{eq:pullback_on_Hn+1}
        \Hup^{n+1}(\V;A) \to \Hup^{n+1}(\Ubf;\piupperstar A)
    \end{equation}
    is an isomorphism.
    Since the isomorphism \eqref{eq:pullback_on_Hn+1} carries $ \alpha_{G} $ to $ \alpha_{\piupperstar G} $ and $ \alpha_G \neq 0 $, we conclude that $ \alpha_{\piupperstar G} \neq 0 $ as required.
\end{proof}

\begin{proof}[Proof that \enumref{prop:comparing_global_sections}{1} $ + $ \enumref{prop:comparing_global_sections}{4} $ \Rightarrow $ \enumref{prop:comparing_global_sections}{5}]
    Using the fact that every object of $ \V' $ is truncated, we proceed by induction on the integer $ n \geq 1 $ such that $ F $ is $ n $-truncated.
    The base case $ n = 1 $ is satisfied by assumption \enumref{prop:comparing_global_sections}{1}.
    For the induction step, assume that we know the claim for $ n $-truncated objects of $ \V' $, and let $ F \in \V' $ be an $ (n+1) $-truncated object.
    Since pullback functors commute with $ n $-truncations, we have a commutative square
    \begin{equation}\label{sq:truncation_need_pullback}
        \begin{tikzcd}[column sep=4.5em]
            \Gamma(\V; F) \arrow[r, "c_{F}"] \arrow[d] & \Gamma(\Ubf; \piupperstar(F))  \arrow[d] \\
            \Gamma(\V; \trun_{\leq n} F) \arrow[r, "c_{\trun_{\leq n} F}"'] & \Gamma(\Ubf; \piupperstar(\trun_{\leq n}F)) \period
        \end{tikzcd}  
    \end{equation}
    By the inductive hypothesis, the morphism $ c_{\trun_{\leq n} F} $ is an equivalence.
    Hence it suffices to show that the square \eqref{sq:truncation_need_pullback} is a pullback square.
    If $ \Gamma(\V; \trun_{\leq n} F) = \emptyset $, then by the inductive hypotheses all spaces appearing in \eqref{sq:truncation_need_pullback} are empty, so this is clear.
    So assume that $ \Gamma(\V; \trun_{\leq n} F) \neq \emptyset $; then we need to show that for every point of $ \Gamma(\V; \trun_{\leq n} F) $, the induced map on fibers
    \begin{equation*}
        \begin{tikzcd}[column sep=4.5em]
            \fib\paren{\Gamma(\V; F) \to \Gamma(\V; \trun_{\leq n} F)} %
            \arrow[r, "\theta_{F,n}"] & %
            \fib\paren{\Gamma(\Ubf; \piupperstar F) \to \Gamma(\Ubf; \piupperstar\trun_{\leq n} F))}
        \end{tikzcd}
    \end{equation*}
    is an equivalence.

    Given such a point $ \fromto{\ast}{\trun_{\leq n} F} $ of $ \Gamma(\V; \trun_{\leq n} F) $, write 
    \begin{equation*}
        \trun_{\geq n+1} F \colonequals \fib(F \to \trun_{\leq n} F) \period
    \end{equation*}
    Since $ F $ is $ (n+1) $-truncated, $ \trun_{\geq n+1} F $ is $ (n+1) $-truncated and $ n $-connected.
    That is, $ \trun_{\geq n+1} F $ is an $ (n+1) $-gerbe.
    Since the global section functors and pullback functors commute with finite limits, we see that there is a commutative square
    \begin{equation*}
        \begin{tikzcd}[column sep=5em]
            \Gamma(\V;\trun_{\geq n+1} F) \arrow[d] \arrow[r, "\sim"{yshift=-0.25em}] & \fib\paren{\Gamma(\V; F) \to \Gamma(\V; \trun_{\leq n} F)} \arrow[d, "\theta_{F,n}"] \\ 
            \Gamma(\Ubf;\piupperstar \trun_{\geq n+1} F) \arrow[r, "\sim"{yshift=-0.25em}] & \fib\paren{\Gamma(\Ubf; \piupperstar(F)) \to \Gamma(\Ubf; \piupperstar\trun_{\leq n} F))} \comma
        \end{tikzcd}
    \end{equation*}
    where the horizontal maps are equivalences and the left-hand vertical map is the natural map.
    Since $ \trun_{\geq n+1} F $ is an $ (n+1) $-gerbe, assumption \enumref{prop:comparing_global_sections}{4} implies that the left-hand vertical map is an equivalence.
    Thus $ \theta_{F,n} $ is also an equivalence, as desired.
\end{proof}

To apply \Cref{prop:comparing_global_sections}, we first axiomatize the properties that \Sigmatorsion sheaves satisfy.

\begin{definition}\label{def:etale_coefficient_system}
    Let $ X $ be a qcqs scheme.
    A \defn{étale coefficient subcategory} is a full subcategory $ \Scal(X) \subset  X_{\et} $ satisfying the following properties:
    \begin{enumerate}[label=\stlabel{def:etale_coefficient_system}, ref=\arabic*]
        \item\label{def:etale_coefficient_system.1} The subcategory $ \Scal(X) \subset X_{\et} $ is closed under finite limits.

        \item\label{def:etale_coefficient_system.2} Every object of $ \Scal(X) $ is truncated.

        \item\label{def:etale_coefficient_system.3} For each integer $ n \geq 0 $ and object $ F \in \Scal(X) $, we have $ \trun_{\leq n}(F) \in \Scal(X) $.

        \item\label{def:etale_coefficient_system.4} For each integer $ n \geq 2 $, and $ n $-gerbe $ G \in \Scal(X) $, the banding of $ G $ is in $ \Scal(X) $.
    \end{enumerate}

    An \defn{étale coefficient system} $ \Scal \colon \fromto{\Schqcqsop}{\Catinfty} $ is a subfunctor of the functor
    \begin{equation*}
        (-)_{\et} \colon \fromto{\Schqcqsop}{\Catinfty}
    \end{equation*}
    such that for each qcqs scheme $ X $, the subcategory $ \Scal(X) \subset X_{\et} $ is an étale coefficient subcategory.
\end{definition}

\begin{example}\label{ex:etale_coefficient_systems}
    In light of \Cref{prop:pi_n_of_torsion_gerbe_torsion}, the following are étale coefficient systems:
    \begin{enumerate}[label=\stlabel{ex:etale_coefficient_systems}, ref=arabic*]
        \item\label{ex:etale_coefficient_systems.1} $ \Scal(X) \colonequals X_{\et,<\infty} $ is the \category of truncated étale sheaves on $ X $.

        \item\label{ex:etale_coefficient_systems.2} $ \Sigma $ is a set of primes and $ \Scal(X) $ is the \category of \Sigmatorsion étale sheaves on $ X $.
    \end{enumerate}
\end{example}

The following are the key consequences of \Cref{prop:comparing_global_sections}:

\begin{corollary}\label{cor:global_sections_over_local_rings}
    Let $ \pi \colon \fromto{U}{V} $ be a morphism of qcqs schemes.
    Let $ \Scal(V) \subset V_{\et} $ be an étale coefficient subcategory.
    Assume that the following conditions are satisfied:
    \begin{enumerate}[label=\stlabel{cor:global_sections_over_local_rings}, ref=\arabic*]
        \item\label{cor:global_sections_over_local_rings.1} If $ G \in \Scal(V) $ is $ 1 $-truncated, then the natural map $ \fromto{\Gammaet(V;G)}{\Gammaet(U;\piupperstar G)} $ is an equivalence.

        \item\label{cor:global_sections_over_local_rings.2} For each abelian group object $ A $ of $ \Scal(V)_{\leq 0} $ and integer $ i \geq 0 $, the natural map
        \begin{equation*}
            \fromto{\Het^i(V;A)}{\Het^i(U;\piupperstar A)}
        \end{equation*}
        is an isomorphism.  
    \end{enumerate}
    Then for each $ F \in \Scal(V) $, the natural map $ \fromto{\Gammaet(V;F)}{\Gammaet(U;\piupperstar F)} $ is an equivalence.
\end{corollary}

\begin{proof}
    The claim follows from \Cref{prop:comparing_global_sections}; in light of \Cref{prop:comparing_global_sections.2_via_cohomology}, hypotheses \enumref{prop:comparing_global_sections}{1}--\enumref{prop:comparing_global_sections}{3} are satisfied by our assumptions and the definition of a coefficient subcategory.
\end{proof}

\begin{definition}
    Let $ \Scal \colon \fromto{\Schqcqsop}{\Catinfty} $ be an étale coefficient system.
    We say that a commutative square of qcqs schemes
    \begin{equation}\label{sq:general_square_2}
        \begin{tikzcd}
            W \arrow[r, "\fbar"] \arrow[d, "\gbar"'] & Y \arrow[d, "g"] \\ 
            X \arrow[r, "f"'] & Z 
        \end{tikzcd}
    \end{equation}
    \defn{satisfies basechange with $ \Scal $-coefficients} if for each $ F \in \Scal(Y) $, the exchange transformation
    \begin{equation*}
        \fromto{\fupperstar\glowerstar(F)}{\gbarlowerstar\fbarupperstar(F)}
    \end{equation*}
    is an equivalence.
\end{definition}

\begin{corollary}[(reducing to strictly henselian local rings)]\label{cor:reduction_to_1-truncated_local}
    Let $ \Scal \colon \fromto{\Schqcqsop}{\Catinfty} $ be an étale coefficient system and consider a commutative square \eqref{sq:general_square_2} of qcqs schemes.
    Assume that for each geometric point $ \fromto{x}{X} $ with image $ \fromto{z}{Z} $, the following conditions are satisfied:
    \begin{enumerate}[label=\stlabel{cor:reduction_to_1-truncated_local}, ref=\arabic*]
        \item\label{cor:reduction_to_1-truncated_local.1} If $ G \in \Scal(\Loc{Y}{z}) $ is $ 1 $-truncated, then the natural map $ \fromto{\Gammaet(\Loc{Y}{z};G)}{\Gammaet(\Loc{W}{x};\pbarupperstar G)} $ is an equivalence.

        \item\label{cor:reduction_to_1-truncated_local.2} For each abelian group object $ A $ of $ \Scal(\Loc{Y}{z})_{\leq 0} $ and integer $ i \geq 0 $, the natural map
        \begin{equation*}
            \fromto{\Het^i(\Loc{Y}{z};A)}{\Het^i(\Loc{W}{x};\pbarupperstar A)}
        \end{equation*}
        is an isomorphism.  
    \end{enumerate}
    Then the square \eqref{sq:general_square_2} satisfies basechange with $ \Scal $-coefficients.
\end{corollary}

\begin{proof}
    Combine \Cref{cor:when_is_the_exchange_transformation_an_equivalence,cor:global_sections_over_local_rings}.
\end{proof}

\begin{remark}\label{rem:arguments_work_for_general_topoi}
    The results of \cref{subsec:generalities_on_strictly_henselian_local_rings,subsec:stalk_of_the_exchange_transformation,subsec:reduction_to_the_local_case} hold with étale \topoi of schemes replaced by arbitrary \topoi.
    To formulate these results in this more general setting, one replaces `geometric point' by `point of \atopos', `étale \topos of the spectrum of a strictly henselian local ring' with `local \topos' (see \cites[Exposé VI, 8.4.6]{MR50:7131}[\S6.2]{arXiv:1807.03281}[\S C.3.6]{MR2063092}{MR977478}), and the `qcqs' assumption by the assumption that the \topos is `bounded coherent' (see \cite[Definitions \SAGthmlink{A.2.0.12} \& \SAGthmlink{A.7.1.2}]{SAG}).
    We have taken care to write the proofs so that they work \textit{verbatim} in this more general setting.
    However, in order to keep the arguments reasonably familiar to an algebro-geometric audience, we decided to formulate the results of this section for étale \topoi of schemes.
\end{remark}


\subsection{Nonabelian basechange theorems}\label{subsec:nonabelian_basechange_theorems}

We now use the results of \cref{subsec:reduction_to_the_local_case} to deduce a number of nonabelian basechange theorems from results already available in the literature.
The first two are the nonabelian refinements of the smooth and proper basechange theorems.

\begin{recollection}\label{def:prosmooth_morphism}
    A morphism of schemes $ f \colon \fromto{X}{Z} $ is \defn{\prosmooth} if there exists a cofiltered diagram $ X_{\bullet} \colon \fromto{I}{\Sch_{Z}} $ of smooth $ Z $-schemes with affine transition maps such that $ X \isomorphic \lim_{i \in I} X_i $ and $ f $ is the projection.
\end{recollection}

\begin{example}
    Let $ f \colon \fromto{X}{Z} $ be a \prosmooth morphism of schemes, and let $ \fromto{x}{X} $ be a geometric point with image $ \fromto{z}{Z} $.
    Then the induced morphism on spectra of strictly local schemes $ \fromto{\Loc{X}{x}}{\Loc{Z}{z}} $ is \prosmooth.
\end{example}

For the next two results, let
\begin{equation}\label{square:nonabelian_basechange}
    \begin{tikzcd}
        W \arrow[r, "\fbar"] \arrow[d, "\gbar"'] \arrow[dr, phantom, very near start, "\lrcorner", xshift=-0.25em, yshift=0.25em] & Y \arrow[d, "g"] \\ 
        X \arrow[r, "f"'] & Z 
    \end{tikzcd}
\end{equation}
be a pullback square of qcqs schemes.

\begin{corollary}[(nonabelian smooth basechange)]\label{cor:nonabelian_smooth_basechange}
    Write $ \Sigma $ for the set of primes invertible on $ Z $.
    If $ f $ is \prosmooth, then the pullback square \eqref{square:nonabelian_basechange} satisfies basechange with \Sigmatorsion coefficients.
\end{corollary}

\begin{proof}
    It suffices to show that the étale coefficient system of \Sigmatorsion sheaves satisfies the hypotheses of \Cref{cor:reduction_to_1-truncated_local}.
    Hypothesis \enumref{cor:reduction_to_1-truncated_local}{1} follows from Giraud's smooth basechange for sheaves of groupoids \cite[Chapitre VII, Théorème 2.1.2]{MR0344253}.
    Hypothesis \enumref{cor:reduction_to_1-truncated_local}{2} is the classical smooth basechange theorem \cites[Exposé XII, Corollaire 1.2]{MR50:7132}.
\end{proof}

\begin{corollary}[(nonabelian proper basechange)]\label{cor:nonabelian_proper_basechange}
    If $ g $ is proper, then the pullback square \eqref{square:nonabelian_basechange} satisfies basechange with torsion coefficients.
\end{corollary}

\begin{proof}
    It again suffices to show that the étale coefficient system of torsion sheaves satisfies the hypotheses of \Cref{cor:reduction_to_1-truncated_local}.
    Hypothesis \enumref{cor:reduction_to_1-truncated_local}{1} follows from Giraud's proper basechange for sheaves of groupoids \cites[Chapitre VII, Théorème 2.2.2]{MR0344253}[Exposé XX, Théorème 2.1.2]{MR3309086},%
    \footnote{Giraud's result makes noetherianity asssumptions; \cite[Exposé XX, Théorème 2.1.2]{MR3309086} explains why these assumptions are unnecessary.} 
    and hypothesis \enumref{cor:reduction_to_1-truncated_local}{2} is the classical result \cites[Exposé XII, Théorème 5.1]{MR50:7132}.
\end{proof}

\begin{remark}\label{rem:Choughs_work}
    Though the specific reduction to a more simple case differs, the key idea in our proof of nonabelian proper basechange is the same as in Chough's proof \cite[Theorem 1.2]{MR4493612}.
    A combination of Chough's work and the proof of the basechange theorem for oriented fiber products of bounded coherent \topoi \cites[Theorem 7.1.7]{arXiv:1807.03281}[Exposé XI, Théorème 2.4]{MR3309086} inspired our proofs of \Cref{prop:comparing_global_sections,cor:nonabelian_smooth_basechange,cor:nonabelian_proper_basechange}.
\end{remark}

Now we explain the nonabelian extensions of the Gabber--Huber affine analogue of the proper basechange theorem \cites{MR1286833}{MR1214956} and the Fujiwara--Gabber rigidity theorem \cite[Corollary 6.6.4]{MR1360610}.%
\footnote{The Fujiwara--Gabber theorem generalizes a result of Elkik \cite[p. 579]{MR345966}.}
To state these results, we fix the following notation.

\begin{notation}\label{ntn:henselian_pairs}
    Let $ (A,I) $ be a henselian pair and $ f \colon \fromto{X}{\Spec(A)} $ a proper morphism.
    Write
    \begin{equation*}
        Z \colonequals \Spec(A/I) \crosslimits_{\Spec(A)} X \comma
    \end{equation*}
    and write $ i \colon \incto{Z}{X} $ for the inclusion.
    Let $ U \subset \Spec(A) $ be an open containing $ \Spec(A) \sminus \Vup(I) $.
    Write $ A\Icomp $ for the $ I $-adic completion of $ A $, and write
    \begin{equation*}
        \Uhat \colonequals U \crosslimits_{\Spec(A)} \Spec(A\Icomp) \period
    \end{equation*}
    Write $ \pi \colon \fromto{\Uhat}{U} $ for the projection.
\end{notation}

\begin{corollary}[(nonabelian affine analogue of proper basechange)]\label{cor:affine_analogue_of_proper_basechange}
    Let $ (A,I) $ be a henselian pair, and keep \Cref{ntn:henselian_pairs}.   
    Then: 
    \begin{enumerate}[label=\stlabel{cor:affine_analogue_of_proper_basechange}, ref=\arabic*]
        \item\label{cor:affine_analogue_of_proper_basechange.1} For every torsion sheaf of spaces $ F \in X_{\et} $, the natural map $ \Gammaet(X; F) \to \Gammaet(Z; \iupperstar F) $ is an equivalence.

        \item\label{cor:affine_analogue_of_proper_basechange.2} The induced map of profinite spaces $  \Pietprofin(Z) \to \Pietprofin(X) $ is an equivalence.
    \end{enumerate}
\end{corollary}

\begin{proof}
    First note that \enumref{cor:affine_analogue_of_proper_basechange}{2} follows from \enumref{cor:affine_analogue_of_proper_basechange}{1} by restricting to constant sheaves.
    For \enumref{cor:affine_analogue_of_proper_basechange}{1}, by \Cref{cor:global_sections_over_local_rings} applied to the morphism $ i $, it suffices to prove the claim when $ F $ is $ 1 $-truncated, as well for abelian cohomology with torsion coefficients.
    These results are the content of \cite[\S5, Corollary 1]{MR1286833}. 
\end{proof}

\begin{remark}
    The most typical formulation of the affine analogue of proper basechange assumes that $ X = \Spec(A) $ and $ Z = \Spec(A/I) $.
\end{remark}

The following removes the noetherianity and characteristic $ 0 $ assumptions from \cite[Theorem 4.2.2]{MR3427263}.
See also \cite[\S\S6.2 \& 6.3]{MR3714509}.

\begin{corollary}[(nonabelian Fujiwara--Gabber rigidity)]\label{cor:Fujiwara-Gabber_rigidity}
    Let $ (A,I) $ be a henselian pair with $ I \subset A $ finitely generated, and keep \Cref{ntn:henselian_pairs}.
    Then:
    \begin{enumerate}[label=\stlabel{cor:Fujiwara-Gabber_rigidity}, ref=\arabic*]
        \item\label{cor:Fujiwara-Gabber_rigidity.1} For every torsion sheaf of spaces $ F \in U_{\et} $, the natural map \smash{$ \Gammaet(U; F) \to \Gammaet(\Uhat; \piupperstar F) $} is an equivalence.

        \item\label{cor:Fujiwara-Gabber_rigidity.2} The induced map of profinite spaces $ \Pietprofin(\Uhat) \to \Pietprofin(U) $ is an equivalence.
    \end{enumerate}
\end{corollary}

\begin{proof}
    Again, \enumref{cor:Fujiwara-Gabber_rigidity}{2} follows from \enumref{cor:Fujiwara-Gabber_rigidity}{1} by restricting to constant sheaves.
    For \enumref{cor:Fujiwara-Gabber_rigidity}{1}, by \Cref{cor:global_sections_over_local_rings} applied the morphism $ \pi $, it suffices to prove the claim when $ F $ is $ 1 $-truncated, as well for abelian cohomology with torsion coefficients.
    The $ 1 $-truncated case is the content of \cite[Exposé XX, Théorème 2.1.2]{MR3309086}. 
    The abelian cohomology statement is well-known; see, for example, \cite[Theorem 6.11]{MR4278670}.
\end{proof}

\begin{remark}
    The most typical formulation of the Fujiwara--Gabber theorem assumes that
    \begin{equation*}
        U = \Spec(A) \sminus \Vup(I) \andeq \Uhat = \Spec(A\Icomp) \sminus \Vup(I A\Icomp) \period
    \end{equation*} 
\end{remark}

\begin{remark}[{(basechange for torsion sheaves of spectra)}]
   \Cref{cor:nonabelian_smooth_basechange,cor:nonabelian_proper_basechange,cor:affine_analogue_of_proper_basechange,cor:Fujiwara-Gabber_rigidity} imply the same results for ($ \Sigma $-)torsion sheaves of \textit{spectra}.
    Let us briefly explain how.
    Let $ \Sigma $ be a set of prime numbers.
    We say that a spectrum $ E $ is \defn{\Sigmafinite} if for each $ n \in \ZZ $, the space $ \Omega^{\infty} E[n] $ is a \Sigmafinite space.
    (Note that a \Sigmafinite spectrum is necessarily bounded-above.)
    Given this, if $ X $ is a qcqs scheme, we can modify \Cref{def:Sigma-torsion_lisse_sheaf,def:Sigma-torsion_etale_sheaf} to make sense for étale sheaves of spectra.
    Let $ F $ be an étale sheaf of spectra on $ X $. 
    We say that $ F $ is \defn{\Sigmatorsion lisse} if there is an étale cover $ \{U_1,\ldots,U_n\} $ of $ X $ such that for each $ i $, the restriction $ \restrict{F}{U_i} $ is constant with value a \Sigmafinite spectrum.
    We say that $ F $ is \defn{\Sigmatorsion constructible} if there exists a finite poset $P$ and a stratification \smash{$ \{X_p\}_{p \in P} $} of $ X $ by qcqs locally closed subschemes such that for each $ p \in P $, the étale sheaf of spectra \smash{$ \restrict{F}{X_p} $} is \Sigmatorsion lisse.
    We say that $ F $ is \defn{\Sigmatorsion} if $ F $ is truncated and $ F $ can be written as the colimit of a filtered diagram of \Sigmatorsion constructible étale sheaves of spectra.

    Write $ \Stab(X_{\et}) $ for the stable \category of étale sheaves of spectra on $ X $.
    The \category $ \Stab(X_{\et}) $ coincides with the stabilization of the \topos $ X_{\et} $.
    By the functoriality of stabilization in left exact functors, pullback along a morphism of schemes commutes with shifts and the underlying space functor $ \Omega_X^{\infty} \colon \fromto{\Stab(X_{\et})}{X_{\et}} $.
    Moreover, the functor $ \Omega_{X}^{\infty} $ preserves sifted colimits \HTT{Corollary}{5.2.6.18}.
    Hence if $ F \in \Stab(X_{\et}) $ is a \Sigmatorsion lisse, \Sigmatorsion constructible, or \Sigmatorsion étale sheaf of spectra, then for each $ n \in \ZZ $, the étale sheaf of spaces $ \Omega_X^{\infty} F[n] $ has the same property.
    As a result, a direct application of  \cite[Proposition 4.7]{arXiv:2108.03545} shows that \Cref{cor:nonabelian_smooth_basechange,cor:nonabelian_proper_basechange,cor:affine_analogue_of_proper_basechange,cor:Fujiwara-Gabber_rigidity} imply the same results for ($ \Sigma $-)torsion sheaves of spectra.
\end{remark}


\subsection{Application: invariance under specialization}\label{subsec:invariance_under_specialization}

As an immediate application of the results of \cref{subsec:nonabelian_basechange_theorems}, we see that for a proper morphism $ \fromto{X}{S} $, the profinite étale homotopy types of the geometric fiber $ X_{s} $ and the `Milnor ball' $ \Loc{X}{s} $ agree:

\begin{corollary}\label{cor:comparison_homotopy_type_of_fibre_and_Milnor_ball}
    Let $ f \colon \fromto{X}{S}$ be a proper morphism between qcqs schemes and $\fromto{s}{S}$ a geometric point.
    Then the closed immersion $ i \colon \incto{X_{s}}{\Loc{X}{s}} $ induces an equivalence
    \begin{equation*}
        \Pietprofin(X_{s}) \equivalence \Pietprofin(\Loc{X}{s}) \period
    \end{equation*}
\end{corollary}

\begin{proof}
    Apply \Cref{cor:affine_analogue_of_proper_basechange} to the henselian pair $ (\Ocal_{S,s}^{\sh},\mfrak_s) $ and the morphism $ \fromto{\Loc{X}{s}}{\Loc{S}{s}} $.
\end{proof}

\begin{remark}
    \Cref{cor:comparison_homotopy_type_of_fibre_and_Milnor_ball} removes the noetherianity hypotheses from \cite[Proposition 8.6]{MR676809}.
\end{remark}

There is a dual version of \cref{cor:comparison_homotopy_type_of_fibre_and_Milnor_ball} for prosmooth morphisms.
For this, we need a few lemmas.

\begin{lemma}\label{lem:constant_sheaves_on_irreducible_spaces}
    Let $ T $ be an irreducible topological space with generic point $ \eta $.
    Also write $ \eta \colon \incto{\{\eta\}}{T} $ for the inclusion of the generic point.
    Then:
    \begin{enumerate}[label=\stlabel{lem:constant_sheaves_on_irreducible_spaces}, ref=\arabic*]
        \item\label{lem:constant_sheaves_on_irreducible_spaces.1} There is a natural identification $ \etalowerstar = \Gammaupperstar_{\Sh(T)} $ of functors $ \fromto{\Spc}{\Sh(T)} $.

        \item\label{lem:constant_sheaves_on_irreducible_spaces.2} If $ F \in \Sh(T) $ is a constant sheaf, then the unit $ \fromto{F}{\etalowerstar\etaupperstar(F)} $ is an equivalence.
    \end{enumerate}
\end{lemma}

\begin{proof}
    First note that \enumref{lem:constant_sheaves_on_irreducible_spaces}{2} is an immediate consequence of \enumref{lem:constant_sheaves_on_irreducible_spaces}{1}.
    For \enumref{lem:constant_sheaves_on_irreducible_spaces}{1}, notice that since $ T $ is irreducible and $ \eta $ is the generic point, every nonempty open of $ T $ contains $ \eta $.
    Thus, by the definition of the pushforward, we see that for each $ K \in \Spc $ and open $ U \subset T $, we have
    \begin{equation*}
        \etalowerstar(K)(U) =
        \begin{cases}
            K, & U \neq \emptyset \\ 
            \pt, & U = \emptyset \period
        \end{cases}
    \end{equation*}
    Since $ \etalowerstar(K) $ is a sheaf whose restriction to nonempty opens of $ T $ is the constant \textit{presheaf}, we deduce that $ \etalowerstar(K) $ is the constant sheaf at $ K $.
\end{proof}

\begin{recollection}[(schemes with strictly henselian local rings)]\label{rec:everywhere_strictly_local_schemes}
    A scheme $ X $ is \defn{everywhere strictly local} if for each point $ x \in X $, the local ring $ \Ocal_{X,x} $ is strictly henselian.
    In this case, all residue fields of $ X $ are separably closed.
    Moreover, the natural geometric morphism of \topoi
    \begin{equation*}
        \fromto{X_{\et}}{X_{\zar}}
    \end{equation*}
    is an equivalence \cite[Corollary 2.5]{MR3649361}.
\end{recollection}

The following result can be seen as a dual of \cref{prop:stalks_via_global_sections_for_strictly_henselian_local_rings}.

\begin{lemma}\label{lemma:pushforward_from_generic_is_fully_faithful}
	Let $ S $ be an irreducible everywhere strictly local scheme with generic point $ \eta $.
    Note that $ \eta $ is a geometric point, and also write $ \eta \colon \Spec(\upkappa(\eta)) \to S $ for the inclusion of the generic point.
	Then: 
    \begin{enumerate}[label=\stlabel{lemma:pushforward_from_generic_is_fully_faithful}, ref=\arabic*]
        \item\label{lemma:pushforward_from_generic_is_fully_faithful.1} The functor $ \etalowerstar \colon \Spec(\upkappa(\eta))_{\et} \to S_{\et} $ is equivalent to the constant sheaf functor.

        \item\label{lemma:pushforward_from_generic_is_fully_faithful.2} If $ F \in S_{\et} $ is a constant sheaf, then the unit $ \fromto{F}{\etalowerstar\etaupperstar(F)} $ is an equivalence.
    \end{enumerate}
\end{lemma}

\begin{proof}
	Since $ S $ is everywhere strictly local, the natural geometric morphism $ \fromto{S_{\et}}{S_{\zar}} $ is an equivalence.
    So it suffices to prove the claim where we replace étale \topoi by Zariski \topoi. 
    By assumption, the underlying topological space of $ S $ is irreducible.
    Hence the claim is an immediate consequence of \Cref{lem:constant_sheaves_on_irreducible_spaces}.
\end{proof}

\noindent Here is the promised dual of \Cref{cor:comparison_homotopy_type_of_fibre_and_Milnor_ball}:

\begin{proposition}\label{prop:comparison_homotopy_fiber_and_Milnor_coball}
	Let $ S $ be an irreducible everywhere strictly local scheme, write $ \Sigma $ for the set of prime numbers invertible on $ S $, and denote the inclusion of the generic point by $ \eta \colon \Spec(\upkappa(\eta)) \to S $.
    Let $ f \colon X \to S $ be a prosmooth morphism, and write $ i \colon X_\eta \to X $ for the inclusion of the generic (geometric) fiber of $ f $. 
	Then:
	\begin{enumerate}[label=\stlabel{prop:comparison_homotopy_fiber_and_Milnor_coball}, ref=\arabic*]
		\item\label{prop:comparison_homotopy_fiber_and_Milnor_coball.1} For every \emph{constant} \Sigmatorsion étale sheaf $ F $ on $ X$, the unit $  \fromto{F}{i_*i^*(F)} $ is an equivalence.
		
		\item\label{prop:comparison_homotopy_fiber_and_Milnor_coball.2} For every \emph{constant} \Sigmatorsion étale sheaf $ F $ on $ X $, the natural map
		\begin{equation*}
			\fromto{\Gammaet(X_\eta;F)}{\Gammaet(X; \iupperstar F)}
		\end{equation*}
		is an equivalence.
		
		\item\label{prop:comparison_homotopy_fiber_and_Milnor_coball.3} The morphism $ i $ induces an equivalence $ \equivto{\Piet(X_\eta)\Sigmacomp}{\Piet(X)\Sigmacomp} $.
	\end{enumerate}
\end{proposition}

\begin{proof}
	First note that \enumref{prop:comparison_homotopy_fiber_and_Milnor_coball}{1} immediately implies \enumref{prop:comparison_homotopy_fiber_and_Milnor_coball}{2} and \enumref{prop:comparison_homotopy_fiber_and_Milnor_coball}{3}.
    Since every constant étale sheaf on $ X $ is pulled back from $ S $, write $ F \equivalent \fupperstar(G) $ for a constant étale sheaf $ G $ on $ S $.
	We need to show that the unit map
	\begin{equation*}
	   \fupperstar(G) \to \ilowerstar \iupperstar \fupperstar(G) \equivalent \ilowerstar \fupperstar_{\eta} \etaupperstar(G)
	\end{equation*}
	is an equivalence.
	Since $ \etaupperstar(G) $ is a \Sigmatorsion sheaf and $ f $ is prosmooth, applying nonabelian smooth basechange (\Cref{cor:nonabelian_smooth_basechange}) to the pullback square
	\begin{equation*}
        \begin{tikzcd}
    		X_{\eta} \arrow[d, "i"'] \arrow[r, "f_{\eta}"] &  \Spec(\upkappa(\eta)) \arrow[d, "\eta"] \\
    		X \arrow[r, "f"'] & S
    	\end{tikzcd}
    \end{equation*}
	shows that 
    \begin{equation*}
        \ilowerstar \fupperstar_{\eta} \etaupperstar(G) \equivalent \fupperstar \etalowerstar \etaupperstar(G) \period
    \end{equation*}
    We deduce that the map 
    \begin{equation*}
       \fupperstar(G) \to \ilowerstar \iupperstar \fupperstar(G) \equivalent \fupperstar \etalowerstar \etaupperstar(G)
    \end{equation*}
    in question may be identified with applying $ \fupperstar $ to the unit map $ G \to \etalowerstar \etaupperstar(G) $.
	Since $ S $ is everywhere strictly local, \cref{lemma:pushforward_from_generic_is_fully_faithful} shows that the latter map is an equivalence; hence the claim follows.
\end{proof}

We conclude this subsection by using \Cref{cor:comparison_homotopy_type_of_fibre_and_Milnor_ball,prop:comparison_homotopy_fiber_and_Milnor_coball} to show that for a smooth proper morphism $ f \colon \fromto{X}{S} $, the étale homotopy types of the Milnor balls $ \Loc{X}{s} $ and geometric fibers $ X_s $ are invariant under specialization.
First let us recall a bit about étale specializations and define the specialization morphism.

\begin{recollection}[(étale specializations)]\label{rec:etale_specializations}
   Let $ S $ be a scheme and let $ \fromto{s}{S} $ and $ \fromto{t}{S} $ be geometric points.
   An \defn{étale specialization} $ s \cospecializes t $ is a morphism of $ S $-schemes \smash{$ \fromto{\Loc{S}{t}}{\Loc{S}{s}} $}.
   See \stacks{0GJ2} for more background.

   To simplify things, we say `let \smash{$ \alpha \colon \fromto{\Loc{S}{t}}{\Loc{S}{s}} $} be an étale specialization' to mean that the geometric points $ \fromto{s}{S} $ and $ \fromto{t}{S} $ as well as the morphism $ \alpha $ have been specified.
\end{recollection}

\begin{notation}
   Let $ f \colon \fromto{X}{S} $ be a morphism of schemes and let \smash{$ \alpha \colon \fromto{\Loc{S}{t}}{\Loc{S}{s}} $} be an étale specialization.
   We write \smash{$ \alphabar \colon \fromto{\Loc{X}{t}}{\Loc{X}{s}} $} for the basechange of $ \alpha $ along $f$.
\end{notation}

\begin{definition}[(specialization morphism)]
	Let $ S $ be a scheme, let $ \alpha \colon \fromto{\Loc{S}{t}}{\Loc{S}{s}} $ be an étale specialization, and let $ f \colon \fromto{X}{S} $ be a morphism of schemes.
    The morphism $ \alphabar \colon \fromto{\Loc{X}{t}}{\Loc{X}{s}} $ induces a \defn{specialization map}
    \begin{equation*}
       \fromto{\Pietprofin(\Loc{X}{t})}{\Pietprofin(\Loc{X}{s})}
    \end{equation*}
    on the étale homotopy types of the Milnor balls of $ f $.
    If $ f $ is proper, the \defn{specialization map}
    \begin{equation*}
        \spmap_{\alpha} \colon \fromto{\Pietprofin(X_{t})}{\Pietprofin(X_{s})}
    \end{equation*}
    is the unique map making the square
    \begin{equation*}
        \begin{tikzcd}[sep=3em]
            \Pietprofin(X_t) \arrow[r, "\sim"{yshift=-0.25em}] \arrow[d, dotted, "\spmap_{\alpha}"'] & \Pietprofin(\Loc{X}{t}) \arrow[d, "\Pietprofin(\alphabar)"] \\ 
            \Pietprofin(X_s) \arrow[r, "\sim"{yshift=-0.25em}] & \Pietprofin(\Loc{X}{s}) \period
        \end{tikzcd}
    \end{equation*}
    commute.
    Here, the horizontal maps are induced by the inclusions $ \incto{X_t}{\Loc{X}{t}} $ and $ \incto{X_s}{\Loc{X}{s}} $; since $ f $ is proper, \Cref{cor:comparison_homotopy_type_of_fibre_and_Milnor_ball} shows that they are equivalences.
\end{definition}

In order to prove invariance under specialization, we need the following technical observation:

\begin{lemma}\label{lem:making_all_residue_fields_sep_closed}
    Let $ (R,\mfrak) $ be local ring such that the residue field $ R/\mfrak $ is separably closed.
    Let $ K $ be a field and let $ x \colon \Spec(K) \to \Spec(R) $ be a morphism with image $ x \in \Spec(R) $ such that the induced field extension $ \upkappa(x) \subset K $ on residue fields is a separable closure.
    Then there exists a factorization of $ x $ as a composite
    \begin{equation*}
        \begin{tikzcd}
            \Spec(K) \arrow[r, "j"] & \Spec(V) \arrow[r, "\phi"] & \Spec(R)
        \end{tikzcd}
    \end{equation*}
    with the following properties:
    \begin{enumerate}[label=\stlabel{lem:making_all_residue_fields_sep_closed}, ref=\arabic*]
    	\item\label{lem:making_all_residue_fields_sep_closed.1} The ring $ V $ is an everywhere strictly local, local integral domain with maximal ideal $ \pfrak $.

    	\item\label{lem:making_all_residue_fields_sep_closed.2} The map $ \phi $ is induced by a local homomorphism $ \fromto{R}{V} $ and the induced field extension $ \incto{R/\mfrak}{V/\pfrak} $ is an isomorphism.

    	\item\label{lem:making_all_residue_fields_sep_closed.3} The image of $ j $ is the generic point of $ \Spec(V) $ and the induced map $ \Frac(V) \to K $ is an isomorphism.
    \end{enumerate}
\end{lemma}

\begin{proof}
	Let $ \qfrak \subset R $ denote the prime ideal corresponding to the image of $ x $.
	By replacing $ R $ by $ R/\qfrak $, we may assume that $ R $ is a domain and $ x $ has image the generic point of $ \Spec(R) $.
	Thus the extension $ \Frac(R) \to K $ is a separable closure.
	Let $ R' $ be the integral closure of $ R $ in $ K $.
	Then $ R' $ is normal and the induced field extension $ \incto{\Frac(R')}{K} $ is an isomorphism.
	Now pick a prime ideal $ \afrak \subset R' $ lying above the maximal ideal of $ R $.
	Since $ R/\mfrak $ is separably closed, \cite[Proposition 2.6]{MR3649361} shows that the local ring $ V = R'_{\afrak} $ is everywhere strictly local.
	It follows that the factorization
	\begin{equation*}
	   R \to V \to K
	\end{equation*}
	satisfies the desired criteria.
\end{proof}

\begin{proposition}[(invariance under specialization)]\label{prop:invariance_under_specalization}
	Let $ f \colon X \to S $ be a smooth and proper morphism of schemes and let $ \Sigma $ be the set of primes invertible on $ S $. 
	Then for any étale specialization $ \alpha \colon \fromto{\Loc{S}{t}}{\Loc{S}{s}} $ the specialization maps
	\begin{equation*}
	   \Pietprofin(\alphabar) \colon \fromto{\Pietprofin(\Loc{X}{t})}{\Pietprofin(\Loc{X}{s})} \andeq \spmap_{\alpha} \colon \fromto{\Pietprofin(X_{t})}{\Pietprofin(X_{s})}
	\end{equation*}
	become equivalences after \Sigmacompletion.
\end{proposition}

\begin{proof}
    Write $ \alpha_t \colon \fromto{t}{\Loc{S}{s}} $ for the composite
    \begin{equation*}
        \begin{tikzcd}
            t \arrow[r, hooked] & \Loc{S}{t} \arrow[r, "\alpha"] & \Loc{S}{s} \period
        \end{tikzcd}
    \end{equation*}
    By the definition of the specialization maps, it suffices to show that the morphism
    \begin{equation*}
        \begin{tikzcd}
            X_{t} \arrow[r, hooked] & \Loc{X}{t} \arrow[r, "\alphabar"] & \Loc{X}{s}
        \end{tikzcd}
    \end{equation*}
    obtained by pulling back $ \alpha_t $ along $ f \colon \fromto{X}{S} $ induces an equivalence on \Sigmacomplete étale homotopy types.
	Use \Cref{lem:making_all_residue_fields_sep_closed} to factor $ \alpha_t $ to produce a commutative diagram
    \begin{equation*}
        \begin{tikzcd}[sep=3em]
            t \arrow[r, hooked, "j"] \arrow[dr, "\alpha_t"'] & \Spec(V) \arrow[d, "\phi"] & \Spec(V/\mfrak) \arrow[l, hooked', "i_1"'] \arrow[d, "\wr"{xshift=-0.25ex}]\\
            & \Loc{S}{s} & s \comma \arrow[l, "i_2"]
        \end{tikzcd}
    \end{equation*}
	where $ V $, $ j $, and $ \phi $ satisfy properties \enumref{lem:making_all_residue_fields_sep_closed}{1}--\enumref{lem:making_all_residue_fields_sep_closed}{3}, and $ i_1 $ and $ i_2 $ denote the inclusions of the closed points.
	We now apply the functor $ \Pietprofin(-\times_S X) \Sigmacomp $ to the above diagram.
	Since $ X $ is proper over $ S $, by \Cref{cor:comparison_homotopy_type_of_fibre_and_Milnor_ball} this functor inverts $ i_1 $ and $ i_2 $.
    Hence it also inverts $ \phi $.
	Furthermore, since $ \Spec(V) $ is irreducible and everywhere strictly local, $ t $ is the generic point, and $ X $ is smooth over $ S $, \Cref{prop:comparison_homotopy_fiber_and_Milnor_coball} shows that this functor inverts $ j $.
    We conclude that it also inverts $ \alpha_t $, as desired.
\end{proof}

\begin{remark}
    \Cref{prop:invariance_under_specalization} removes numerous hypotheses from \cite[Corollary 12.13]{MR0245577}.
\end{remark}

%% file: content/arc-descent.tex

\section{Application: \arcdescent}\label{sec:arc-descent}

Introduced by Bhatt and Mathew in \cite{MR4278670}, the \textit{\arctopology} is a very fine Grothendieck topology (finer than the \vtopology) on the category of qcqs schemes.
In practice, many invariants that a priori only satisfy étale descent can be shown to satisfy \arcdescent (see \cite[\S5]{MR4278670}).
For example, étale cohomology with torsion coefficients satisfies \arcdescent.
In this section, we prove a nonabelian version of this result: we show that the profinite étale homotopy type
\begin{equation*}
	\Pietprofin \colon \Schqcqs \to \ProSpcfin
\end{equation*}
is a hypercomplete cosheaf for the \arctopology (see \Cref{thm:Pietprofin_satisfies_arc-descent}). 

We quickly recall the relevant definitions in \cref{subsec:reminders_on_the_v-topology_and_arc-topology}.
In \cref{subsec:arc-descent_for_the_etale_homotopy_type}, we prove that $ \Pietprofin $ is a hypercomplete \arccosheaf.
Besides the general machinery developed in \cite{MR4278670}, the key ingredient for our proof is the nonabelian proper basechange theorem.


\subsection{Reminders on the \vtopology and the \arctopology}\label{subsec:reminders_on_the_v-topology_and_arc-topology}

We begin by recalling the definitions of the v- and arc-topologies.

\begin{definition}[(cosheaves)]
	Let $ (\Scal,\tau) $ be \asite and $ \Ccal $ \acategory.
	We say that a functor $ F \colon \fromto{\Scal}{\Ccal} $ is a \defn{$ \tau $-cosheaf} if the functor $ F^{\op} \colon \fromto{\Scal^{\op}}{\Ccal^{\op}} $ is a $ \tau $-sheaf.
	Equivalently, $ F $ is a $ \tau $-cosheaf if $ F $ sends $ \tau $-covering sieves in $ \Scal $ to colimit diagrams in $ \Ccal $.
	We say that $ F $ is a \defn{hypercomplete $ \tau $-cosheaf} if $ F^{\op} $ is a hypercomplete $ \tau $-sheaf.
\end{definition}

\begin{notation}
	For a scheme $ S $, write $ \Schqcqs_S \subset \Sch_S $ for the full subcategory of $ S $-schemes spanned by those $ S $ schemes that are qcqs over $ \Spec(\ZZ) $.
\end{notation}

\begin{recollection}
	A morphism $ f \colon Y \to X $ of qcqs schemes is an \defn{\arccover} if for any valuation ring $ V $ of rank $ \leq 1 $ and any morphism $ \Spec(V) \to X $, there exists a faithfully flat map $ V \to W $ of rank $ \leq 1 $ valuation rings and a morphism $ \Spec(W) \to Y $ that fits into a commutative square
	\begin{equation*}
		\begin{tikzcd}
			\Spec(W) \arrow[r] \arrow[d] & Y \arrow[d] \\ 
			\Spec(V) \arrow[r] & X \period
		\end{tikzcd}
	\end{equation*}
	For a qcqs scheme $ S $, \arccovers generate a topology on the category $ \Schqcqs_{S} $ that we call the \defn{\arctopology}.
	
	Similarly, $ f \from Y \to X $ is called a \defn{\vcover}, if for every valuation ring $ V $ (not necessarily of rank $ \leq 1 $) and morphism $ \Spec(V) \to X $, there is a faithfully flat map of valuation rings $ V \to W $ and commutative square as above.
	The resulting topology on $ \Schqcqs_{S} $ is called the \defn{\vtopology}.
\end{recollection}

\begin{nul}
	Every \vcover is an \arccover \cite[Proposition 2.1]{MR4278670}.
	Also note that by the valuative criterion for properness, every proper surjection of qcqs schemes is a \vcover.
\end{nul}

\begin{nul}
	In general, the \arctopology is strictly finer than the \vtopology \cite[Corollary 2.9]{MR4278670}.
	If $ X $ is noetherian, then every \arccover $ f \colon \fromto{Y}{X} $ is also a \vcover \cite[Proposition 2.6]{MR4278670}.
	Moreover, in this case, $ f $ is also a cover for Voevodsky's \htopology \cites{MR2687724}{MR1403354}.
	See \cite[Theorem 2.8]{MR2679038}.
\end{nul}

Bhatt and Mathew gave a convenient criterion for verifying that a \vsheaf is an \arcsheaf in terms of excision (see \Cref{thm:arc-descent_critera}).
In the remainder of this subsection, we recall the relevant terminology to state this criterion.
In \cref{subsec:arc-descent_for_the_etale_homotopy_type}, we make use of this result to deduce that the profinite étale homotopy type satisfies \arcdescent.

First, this criterion requires the \vsheaf to be \textit{finitary}:

\begin{recollection}\label{rec:finitary_functors}
	Let $ \Ccal $ be \acategory with filtered colimits and let $ S $ be a scheme.
	A functor
	\begin{equation*}
		F \colon \Schqcqsop_S \to \Ccal
	\end{equation*}
	is \defn{finitary} if $ F $ carries limits of cofiltered diagrams of $ S $-schemes with affine transition maps to filtered colimits in $ \Ccal $.

	If $ \Ccal $ is \acategory with cofiltered limits, we say that a functor $ F \colon \Schqcqs_S \to \Ccal $ is \defn{finitary} if the corresponding functor \smash{$ F^{\op} \colon \Schqcqsop_S \to \Ccal^{\op} $} is finitary.
\end{recollection}

\begin{proposition}\label{prop:profinite_etale_homotopy_type_is_finitary}
	The protruncated and profinite étale homotopy types
	\begin{equation*}
		\Pietprotrun \colon \fromto{\Schqcqs}{\ProSpctrun} \andeq \Pietprofin \colon \fromto{\Schqcqs}{\ProSpcfin}
	\end{equation*}
	are finitary functors.
\end{proposition}

\begin{proof}
	First notice that profinite completion preserves cofiltered limits, so it suffices to prove the claim for the protruncated étale homotopy type.
	For this, note that by \cites[Exposé VII, Lemme 5.6]{MR50:7131}[Lemma 3.3]{MR4296353}, the functor
	\begin{equation*}
		(-)_{\et} \colon \Schqcqs \to \RTop_{\infty}
	\end{equation*}
	is finitary.
	By \cite[Corollary 4.3.7]{arXiv:1807.03281} and \SAG{Corollary}{A.8.3.3}, the protruncated shape preserves limits of cofiltered diagrams of bounded coherent \topoi and coherent geometric morphisms.
	The claim now follows from the facts that the étale \topos of a qcqs scheme is bounded coherent \cite[Proposition 3.7.3]{arXiv:1807.03281}, and that for every morphism between qcqs schemes $ f \colon \fromto{X}{Y} $, the induced geometric morphism $ \flowerstar \colon \fromto{X_{\et}}{Y_{\et}} $ is coherent \cite[Example 3.7.5]{arXiv:1807.03281}.
\end{proof}

\noindent A useful fact about finitary functors is that they are determined by their values on finitely presented schemes:

\begin{recollection}[{\stacks{09MV}}]\label{rec:every_qcqs_scheme_over_a_qs_scheme_is_a_cofiltered_limit_of_fp_schemes}
	Let $ S $ be a quasiseparated scheme. 
	Then every object of $ \Schqcqs_S $ can be written as the limit of a cofiltered diagram of finitely presented $ S $-schemes with affine transition maps. 
\end{recollection}

\begin{observation}[(equivalences of finitary functors)]\label{obs:equivalences_of_finitary_functors_can_be_checked_on_finitely_presented_S-schemes}
	Let $ \Ccal $ be \acategory with filtered colimits and let $ S $ be a quasiseparated scheme.
	In light of \Cref{rec:every_qcqs_scheme_over_a_qs_scheme_is_a_cofiltered_limit_of_fp_schemes}, given finitary functors
	\begin{equation*}
		F,G \colon \fromto{\Schqcqsop_S}{\Ccal} \comma
	\end{equation*}
	a natural transformation $ \alpha \colon \fromto{F}{G} $ is an equivalence if and only if $ \alpha $ is an equivalence when restricted to the full subcategory spanned by the finitely presented $ S $-schemes. 
\end{observation}

Second, \arcsheaves automatically satisfy excision for \textit{Milnor squares}:

\begin{recollection}[(Milnor excision)]
	A commutative square of schemes
	\begin{equation*}
		\begin{tikzcd}
			Z \arrow[d] \arrow[r, hooked] & X \arrow[d, "f"] \\ 
			Z' \arrow[r, "i"', hooked] & X'
		\end{tikzcd}
	\end{equation*}
	is a \defn{Milnor square} if it is bicartesian, $ f $ is affine, and $ i $ is a closed immersion.
	Given a scheme $ S $, a functor
	\begin{equation*}
		F \colon \fromto{\Schqcqsop_{S}}{\Ccal}
	\end{equation*}
	satisfies \defn{Milnor excision} if $ F $ carries Milnor squares to pullback squares in $ \Ccal $.
\end{recollection}

\noindent We also recall the following weakening of Milnor excision:

\begin{recollection}[(\aicvexcision)]
	Let $ \Ccal $ be \acategory and $ S $ a scheme.
	A functor
	\begin{equation*}
		F \colon \Schqcqsop_S \to \Ccal
	\end{equation*}
	satisfies \defn{\aicvexcision} if for any absolutely integrally closed valuation ring $ V $ and $ \pfrak \in \Spec(V) $ the square
	\begin{equation*}
		\begin{tikzcd}
			F(\Spec(V)) \arrow[d] \arrow[r] & F(\Spec(V_{\pfrak})) \arrow[d] \\ 
			F(\Spec(V/\pfrak)) \arrow[r] & F(\Spec(\upkappa(\pfrak))
		\end{tikzcd}
	\end{equation*}
	is a pullback in $ \Ccal $.
\end{recollection}

\begin{nul}
	We say that a functor $ F \colon \Schqcqs_S \to \Ccal $ satisfies \defn{Milnor excision} (resp., \defn{\aicvexcision}) if $ F^{\op} $ satisfies Milnor excision (resp., \aicvexcision).
\end{nul}

Third, we need the target \category to be sufficiently nice.

\begin{definition}
	\Acategory $ \Ccal $ is \defn{compactly generated by cotruncated objects} if $ \Ccal $ is compactly generated and every compact object of $ \Ccal $ is cotruncated (i.e., truncated in $ \Ccal^{\op} $).
\end{definition}

\begin{example}
	The \category $ \ProSpcfin^{\op} \equivalent \Ind(\Spcfin^{\op}) $ is compactly generated by cotruncated objects.
\end{example}

\begin{example}
	If $ \Ccal $ is a compactly generated \category, then for each $ n \geq 0 $, the subcategory $ \Ccal_{\leq n} \subset \Ccal $ is compactly generated by cotruncated objects.
\end{example}

Finally, the promised characterization of \arcdescent in terms of \vdescent and excision: 

\begin{theorem}[{\cite[Theorem 4.1]{MR4278670}}]\label{thm:arc-descent_critera}
	Let $ \Ccal $ be an \category that is compactly generated by cotruncated objects, and let $ S $ be a qcqs scheme. 
	Let $ F \colon \Schqcqsop_S \to \Ccal $ be a finitary \vsheaf.
	Then the following are equivalent:
	\begin{enumerate}[label=\stlabel{thm:arc-descent_critera}]
		\item $ F $ is an \arcsheaf.

		\item $ F $ satisfies Milnor excision.

		\item $ F $ satisfies \aicvexcision.
	\end{enumerate}
\end{theorem}


\subsection{arc-descent for the étale homotopy type}\label{subsec:arc-descent_for_the_etale_homotopy_type}

In this subsection, we use \Cref{thm:arc-descent_critera} to prove:

\begin{theorem}\label{thm:Pietprofin_satisfies_arc-descent}
	The functor
	\begin{equation*}
		\Pietprofin (-) \colon \Schqcqs \to \ProSpcfin
	\end{equation*}
	is a finitary hypercomplete \arccosheaf.
	In other words, for any semi-simplicial \archypercovering $ p_\bullet \colon U_\bullet \to X $ the induced diagram $ \Pietprofin (U_\bullet) \to \Pietprofin(X) $ is a colimit diagram in $ \ProSpcfin $.
\end{theorem}

We first verify that the profinite étale homotopy type satisfies \aicvexcision.
This follows from the fact that the \category of constructible étale sheaves of spaces satisfies \aicvexcision.

\begin{recollection}
	If $ V $ is an absolutely integrally closed valuation ring, then every local ring of $ V $ is strictly henselian \cite[Lemma 5.3]{MR4278670}.
	That is, $ \Spec(V) $ is everywhere strictly local in the sense of \Cref{rec:everywhere_strictly_local_schemes}.
\end{recollection}

\begin{notation}
	Given a qcqs scheme $ X $, write $ X_{\et}^{\cons} \subset X_{\et} $ for the full subcategory spanned by the constructible étale sheaves of spaces in the sense of \Cref{def:Sigma-torsion_etale_sheaf}.
\end{notation}

\begin{recollection}[(sheaves on finite posets)]\label{rec:sheaves_on_finite_posets}
	Let $ P $ be a poset.
	Recall that the \defn{Alexandroff topology} on the set $ P $ is the topology where the open subsets are the subsets that are upwards-closed under the partial order.
	If $ P $ is finite, then there is a natural equivalence
	\begin{equation*}
		\Sh(P) \equivalent \Fun(P,\Spc) \period
	\end{equation*}
	See \cite[8.1.1]{exodromy}.
\end{recollection}


\begin{lemma}\label{lem:etale_topoi_of_aic_valuation_rings}
	Let $ V $ be an absolutely integrally closed valuation ring of finite rank $ n $.
	Then there is a natural equivalence of \topoi
	\begin{equation*}
		\Spec(V)_{\et} \equivalent \Fun(\{0 < \cdots < n\},\Spc) \period
	\end{equation*} 
	Moreover, this equivalence restricts to an equivalence
	\begin{equation*}
		\Spec(V)_{\et}^{\cons} \equivalent \Fun(\{0 < \cdots < n\},\Spcfin) \period
	\end{equation*}
\end{lemma}

\begin{proof}
	Since the natural geometric morphism $ \fromto{\Spec(V)_{\et}}{\Spec(V)_{\zar}} $ is an equivalence, it suffices to prove the claim for Zariski \topoi.
	The claim now follows from the fact that the Zariski topological space of $ \Spec(V) $ is isomorphic to the poset $ \{0 < \cdots < n\} $ equipped with the Alexandroff topology and \Cref{rec:sheaves_on_finite_posets}.
\end{proof}

\begin{corollary}\label{cor:constructible_sheaves_satisfy_aic-v-excision}
	The functor $ (-)_{\et}^{\cons} \colon \Schqcqsop_{S} \to \Catinfty $ satisfies \aicvexcision.
\end{corollary}

\begin{proof}
	Since $ (-)_{\et}^{\cons} $ is a finitary functor, by \cite[Lemma 2.22]{MR4278670} it suffices to check \aicvexcision for finite rank absolutely integrally closed valuation rings. 
	So let $ V $ be an absolutely integrally closed valuation ring of finite rank $ n $, let $ \pfrak \in \Spec(V) $, and write $ i $ for the rank of the valuation ring $ V/\pfrak $.
	By \Cref{lem:etale_topoi_of_aic_valuation_rings}, the square
	\begin{equation*}
		\begin{tikzcd}
			\Spec(V)_{\et}^{\cons} \arrow[r] \arrow[d] & \Spec(V_{\pfrak})_{\et}^{\cons} \arrow[d] \\
			\Spec(V/\pfrak)_{\et}^{\cons} \arrow[r] & \Spec(\upkappa(\pfrak))_{\et}^{\cons}
		\end{tikzcd}
	\end{equation*}
	is given by applying $ \Fun(-,\Spcfin) $ to the pushout square 
	\begin{equation*}
		\begin{tikzcd}
			\{i\} \arrow[r, hooked] \arrow[d, hooked] & \{i < \cdots < n\} \arrow[d, hooked] \\
			\{0 < \cdots < i\} \arrow[r, hooked] & \{0 < \cdots < n \}
		\end{tikzcd}
	\end{equation*}
	in $ \Catinfty $.
	Therefore it is a pullback, as desired.
\end{proof}

\begin{proposition}\label{prop:nonabelian_etale_arc_descent}
	Let $ S $ be a qcqs scheme and let $ F \in S_{\et} $ be a constructible sheaf.
	Then the functor
	\begin{equation*}
		\Gammaet(-;F) \colon \Schqcqsop_{S} \to \Spc \, , \; \;	[f\colon X \to S] \mapsto \Gammaet(X;f^*F)
	\end{equation*}
	is a finitary \archypersheaf.
\end{proposition}

\begin{proof}
	Since there exists an integer $ n \geq 0 $ such that $ F $ is $ n $-truncated, $ \Gammaet(-;F) $ takes values in $ \Spc_{\leq n} $; hence it suffices to see that $ \Gammaet(-;F) $ is an \arcsheaf.
	The functor $ \Gammaet(-;F) $ is finitary, so by \Cref{thm:arc-descent_critera} it suffices to see that $ \Gammaet(-;F) $ satisfies \aicvexcision and \vdescent.
	Since 
	\begin{equation*}
		\Gammaet(X;\fupperstar F) = \Map_{X_{\et}^{\cons}}(\ast,\fupperstar F)
	\end{equation*}
	and mapping spaces in pullbacks of \categories are computed as pullbacks of the mapping spaces, \aicvexcision is an immediate consequence of \Cref{cor:constructible_sheaves_satisfy_aic-v-excision}.
	For \vdescent, we note that the proof of \cite[Proposition 5.2]{MR4278670} \textit{verbatim} applies in our situation.
	The only non-geometric input that is used there is the proper basechange theorem; in the present nonabelian situation, we appeal to \Cref{cor:nonabelian_proper_basechange}.
\end{proof}

\begin{proof}[Proof of \Cref{thm:Pietprofin_satisfies_arc-descent}]
	Since colimits in
	\begin{equation*}
		\ProSpcfin \equivalent \Funlex(\Spcfin,\Spc)^{\op}
	\end{equation*}
	are computed as pointwise limits in $ \Funlex(\Spcfin,\Spc) $, the claim is equivalent to showing that if $ F \in X_{\et} $ is a constant sheaf at a \pifinite space, then the natural map
	\begin{equation*}
		\Gammaet(X;F) \to \lim_{[n] \in \Deltainj} \Gammaet(U_n;p_n^* F)
	\end{equation*}
	is an equivalence.
	This follows from \Cref{prop:nonabelian_etale_arc_descent}.
\end{proof}



\begin{remark}[(\arcdescent for \categories of constructible sheaves)]
	Let $ \Lambda $ be a finite ring.
	Bhatt--Mathew showed that the functor $ \goesto{X}{\Dcons(X_{\et};\Lambda)} $ that carries a qcqs scheme to its constructible derived \category is an \archypersheaf \cite[Theorem 5.13]{MR4278670}.
	The key ingredients of the proof are: \aicvexcision, proper basechange, the preservation of constructibility under proper pushforwards, and Lurie's general result about basechange and descent \HA{Corollary}{4.7.5.3}.
	In the nonabelian setting, the only part of this argument that is currently unavailable is that pushforwards along proper morphisms preserve constructibility.
	Once this is proven, it will also follow that the functor $ \goesto{X}{X_{\et}^{\cons}} $ satisfies \archyperdescent.
\end{remark}

%% file: content/Kunneth_formulas.tex

\section{Application: Künneth formulas}\label{sec:Kunneth_formulas}

Let $ k $ be a field of characteristic $ p \geq 0 $.
In this section, we prove Künneth formulas for the étale homotopy type (possibly completed away from $ p $).
For example, if $ k $ is separably closed we first show that for qcqs $ k $-schemes $ X $ and $ Y $ the natural map of prime-to-$ p $ étale homotopy types
\begin{equation*}
	\Piet(X \cross_k Y)\pprimecomp \longrightarrow \Piet(X)\pprimecomp \cross \Piet(Y)\pprimecomp
\end{equation*}
is an equivalence (\Cref{thm:prime-to-p_Kunneth_formula}).
At least when $ X $ and $ Y $ are finite type, this was already proven by Orgogozo in 2003 \cite[Corollaire 4.9]{MR1975807}.
Orgogozo's result seems to not be very widely-known (see \cite{MO:257722}); one goal of this section is to disseminate it.
From this, we derive some relative Künneth formulas over more general fields (\Cref{cor:relative_Kunneth_formula_proper,cor:relative_Kunneth_formula_prime-to-p}).
These also imply Künneth formulas for symmetric powers (see \Cref{rem:symmetric_Kunneth_formula_proper_case,rem:symmetric_Kunneth_formula_prime-to-p}).

In \cref{subsec:Kunneth_formulas_via_basechange}, we start by proving a general result relating nonabelian basechange theorems and Künneth formulas over separably closed fields.
\Cref{subsec:prime-to-p_Kunneth_formulas} uses this result to prove the Künneth formula for the prime-to-$ p $ étale homotopy type over separably closed fields and derives some consequences (e.g., $ \AA^1 $-invariance).
In \cref{subsec:relative_Kunneth_formulas}, we prove relative Künneth formulas over fields that are not separably closed.


\subsection{Künneth formulas via basechange}\label{subsec:Kunneth_formulas_via_basechange}

The purpose of this subsection is to prove the following proposition.
We are most interested in the case where $ W = X \cross_k Y $.

\begin{proposition}[(Künneth formula from basechange)]\label{prop:Kunneth_formula_from_basechange}
	Let $ \Sigma $ be a set of prime numbers, $ k $ a separably closed field, and 
	\begin{equation}\label{sq:span_over_a_separably_closed_field}
	    \begin{tikzcd}
	        W \arrow[r, "\fbar"] \arrow[d, "\gbar"'] & Y \arrow[d, "g"] \\ 
	        X \arrow[r, "f"'] & \Spec(k)
	    \end{tikzcd}
	\end{equation}
	a commutative square of qcqs schemes.
	Assume that for every \emph{constant} \Sigmatorsion étale sheaf $ F $ on $ Y $, the exchange transformation $ \fromto{\fupperstar\glowerstar(F)}{\gbarlowerstar\fbarupperstar(F)} $ is an equivalence.
	Then the natural map of \Sigmacomplete étale homotopy types
	\begin{equation*}
		\longfromto{\Piet(W)\Sigmacomp}{\Piet(X)\Sigmacomp \cross \Piet(Y)\Sigmacomp}
	\end{equation*}
	is an equivalence. 
\end{proposition}

\noindent The proof of \Cref{prop:Kunneth_formula_from_basechange} is an axiomatization of the proof of Chough's Künneth formula in the proper setting \cite[Theorem 5.3]{MR4493612}; see \Cref{ex:Kunneth_formula_proper_case}.
To give the proof, we first recall the basics about the \textit{composition product} of prospaces and its relation to the product.
For this, we make crucial use of the identification of prospaces as left exact accessible functors $ \fromto{\Spc}{\Spc} $ (\Cref{rec:pro-objects_as_left_exact_accessible_functors}).

\begin{recollection}[{\SAG{Remark}{E.2.1.2}}]\label{rec:composition_product}
	Composition of functors defines a monoidal structure $ \goesto{(A,B)}{A \circ B} $ on the \category $ \Fun(\Spc,\Spc)^{\op} $.
	We call this monoidal structure the \defn{composition monoidal structure}.
	Since the composition of two left exact accessible functors $ \fromto{\Space}{\Space} $ is again left exact and accessible, the composition monoidal structure restricts to a monoidal structure on the full subcategory
	\begin{equation*}
		\ProSpc \subset \Fun(\Space,\Space)^{\op} \period
	\end{equation*}
\end{recollection}

\begin{observation}
	The identity functor is both the unit for $ \circ $ and the terminal object of $ \ProSpc $.
	Hence given prospaces $ A $ and $ B $, the universal property of the product provides a natural comparison map
	\begin{equation*}
		c \colon \fromto{A \circ B}{A \times B} \period
	\end{equation*}
	This map is not generally an equivalence.
	However, in the setting of étale homotopy theory, it is close to being an equivalence:
\end{observation}

\begin{example}\label{ex:composition_product_of_etale_homotopy_types}
	Let $ \Sigma $ be a set of prime numbers and let $ X $ and $ Y $ be qcqs schemes.
	By a variant of the proof of \cite[Corollary 2.8.5]{arXiv:1807.03281}, the natural map of prospaces
	\begin{equation*}
		\longfromto{\Piet(X) \of \Piet(Y)}{\Piet(X) \cross \Piet(Y)}
	\end{equation*}
	becomes an equivalence after protruncation, hence also after \Sigmacompletion.
\end{example}

The next two observations relate the composition product to exchange transformations.

\begin{observation}
	Let 
	\begin{equation}\label{eq:span_of_topoi}
	    \begin{tikzcd}
	        \W \arrow[r, "\fbarlowerstar"] \arrow[d, "\gbarlowerstar"'] & \Y \arrow[d, "\Gamma_{\Y,\ast}"] \\ 
	        \X \arrow[r, "\Gamma_{\X,\ast}"'] & \Spc
	    \end{tikzcd}
	\end{equation}
	be a commuative square of \topoi and geometric morphisms.
	Note that the exchange transformation associated to the square \eqref{eq:span_of_topoi} defines a natural transformation
	\begin{equation*}
	    \begin{tikzcd}[sep=5em]
	        \Gamma_{\X,\ast}\Gammaupperstar_{\X} \Gamma_{\Y,\ast}\Gammaupperstar_{\Y} \arrow[r, "\Gamma_{\X,\ast} \Ex\Gammaupperstar_{\Y}"] & \Gamma_{\X,\ast} \gbarlowerstar \fbarupperstar \Gammaupperstar_{\Y} \equivalent \Gamma_{\W,\ast} \Gammaupperstar_{\W} 
	    \end{tikzcd}
	\end{equation*}
	of left exact accessible functors $ \fromto{\Spc}{\Spc} $.
	Let us write
	\begin{equation*}
		\epsilon \colon \longfromto{\Shape(\W)}{\Shape(\X) \of \Shape(\Y)}
	\end{equation*}
	for the corresponding map in $ \ProSpc $.
\end{observation}

\begin{observation}\label{obs:map_to_product_factors_through_composition_product}
	The natural map $ \fromto{\Shape(\W)}{\Shape(\X) \cross \Shape(\Y)} $ factors as a composite
	\begin{equation*}
	    \begin{tikzcd}[sep=3em]
	        \Shape(\W) \arrow[r, "\epsilon"] & \Shape(\X) \of \Shape(\Y) \arrow[r, "c"] & \Shape(\X) \cross \Shape(\Y) \period
	    \end{tikzcd}
	\end{equation*}
\end{observation}

\begin{proof}[Proof of \Cref{prop:Kunneth_formula_from_basechange}]
	Since $ k $ is separably closed, $ \Spec(k)_{\et} \equivalent \Spc $.
	By \Cref{obs:map_to_product_factors_through_composition_product}, the natural map of prospaces $ \fromto{\Piet(W)}{\Piet(X) \cross \Piet(Y)} $ factors as a composite
	\begin{equation*}
		\begin{tikzcd}
			\Piet(W) \arrow[r, "\epsilon"] & \Piet(X) \of \Piet(Y) \arrow[r, "c"] & \Piet(X) \cross \Piet(Y) \period
		\end{tikzcd}
	\end{equation*}
	Since $ \epsilon $ is induced by the exchange transformation associated to the square \eqref{sq:span_over_a_separably_closed_field}, the assumptions imply that the map $ \epsilon $ becomes an equivalence after \Sigmacompletion.
	By \Cref{ex:composition_product_of_etale_homotopy_types}, the map $ c $ also becomes an equivalence after \Sigmacompletion.
	Thus $ c \epsilon $ becomes an equivalence after \Sigmacompletion.
	To conclude, note that \cite[Example 2.17]{Haine:profinite_completions_of_products} shows that the natural map of \Sigmaprofinite spaces
	\begin{equation*}
		\longfromto{(\Piet(X) \cross \Piet(Y))\Sigmacomp}{\Piet(X)\Sigmacomp \cross \Piet(Y)\Sigmacomp}
	\end{equation*}
	is an equivalence.
\end{proof}

We conclude this subsection with two examples.
The first is due to Chough \cite[Theorem 5.3]{MR4493612}; we recapitulate it here for the sake of completeness.

\begin{example}[(Künneth formula, proper case)]\label{ex:Kunneth_formula_proper_case}
	Let $ k $ be a separably closed field and let $ X $ and $ Y $ be qcqs $ k $-schemes.
	If $ Y $ is proper, then the nonabelian proper basechange theorem (\Cref{cor:nonabelian_proper_basechange}) and \Cref{prop:Kunneth_formula_from_basechange} imply that the natural map of profinite spaces
	\begin{equation*}
		\longfromto{\Pietprofin(X \cross_k Y)}{\Pietprofin(X) \cross \Pietprofin(Y)}
	\end{equation*}
	is an equivalence.
\end{example}

\begin{notation}
	Let $ p $ be a prime number or $ 0 $.
	We write $ p' $ for the set of prime numbers \textit{different from} $ p $. 
\end{notation}

\begin{example}[(prime-to-$ p $ Künneth formula, smooth case)]\label{ex:prime-to-p_Kunneth_formula_smooth_case}
	Let $ k $ be a separably closed field of characteristic $ p \geq 0 $, and let $ X $ and $ Y $ be qcqs $ k $-schemes.
	If $ X $ is smooth, then the nonabelian smooth basechange theorem (\Cref{cor:nonabelian_smooth_basechange}) and \Cref{prop:Kunneth_formula_from_basechange} imply that the natural map of \pprimeprofinite spaces
	\begin{equation*}
		\longfromto{\Piet(X \cross_k Y)\pprimecomp}{\Piet(X)\pprimecomp \cross \Piet(Y)\pprimecomp}
	\end{equation*}
	is an equivalence.
\end{example}


\subsection{Prime-to-\texorpdfstring{$ p $}{p} Künneth formulas, after Orgogozo}\label{subsec:prime-to-p_Kunneth_formulas}

In this subsection, we prove the following Künneth formula, removing the smoothness hypothesis from \Cref{ex:prime-to-p_Kunneth_formula_smooth_case}.
We then derive some corollaries.

\begin{theorem}[(prime-to-$ p $ Künneth formula)]\label{thm:prime-to-p_Kunneth_formula}
	Let $ k $ be a separably closed field of characteristic $ p \geq 0 $, and let $ X $ and $ Y $ be qcqs $ k $-schemes.
	Then the natural map of \pprimeprofinite spaces
	\begin{equation*}
		\longfromto{\Piet(X \cross_k Y)\pprimecomp}{\Piet(X)\pprimecomp \cross \Piet(Y)\pprimecomp}
	\end{equation*}
	is an equivalence.
\end{theorem}

\noindent To prove this, we use the fact that every finite type $ k $-scheme admits a \vhypercover by regular $ k $-schemes:

\begin{observation}[(alterations \& \vhypercovers)]\label{rem:regular_v-hypercovers}
	Let $ k $ be a field and $ X $ a finite type $ k $-scheme.
	By the theory of alterations \cites[Theorem 1.1]{Illusie:On_Gabbers_refined_uniformization}[Exposé IX, Théorème 1.1]{MR3309086}[Theorem 4.4]{MR3730515}[Theorem 1.2.5]{MR3665001}, $ X $ admits a proper surjection, in partiuclar a \vcover, from a regular $ k $-scheme.
	By repeatedly applying this result for iterated pullbacks, it follows that there exists a semi-simplicial \vhypercover $ U_\bullet \to X $ where each $ U_n $ is a regular $ k $-scheme.
\end{observation}

\begin{nul}\label{obs:basechange_of_product_is_product_of_basechange}
	Let $ K \supset k $ be an extension of fields, and let $ X $ and $ Y $ be $ k $-schemes.
	Then the natural morphism of $ K $-schemes $ \fromto{(X \cross_k Y)_{K}}{X_{K} \cross_{K} Y_{K}} $ is an isomorphism.
\end{nul}

\begin{proof}[Proof of \Cref{thm:prime-to-p_Kunneth_formula}]
	Let $ \kalg \supset k $ be an algebraic closure of $ k $.
	Since the morphism of schemes $ \fromto{\Spec(\kalg)}{\Spec(k)} $ is a universal homeomorphism and the étale homotopy type is topologically invariant, by basechanging to $ \kalg $ and applying \Cref{obs:basechange_of_product_is_product_of_basechange}, we may assume without loss of generality that $ k $ is algebraically closed.
	Furthermore, since the assignments 
	\begin{equation*}
		X \mapsto \Piet(X \cross_k Y)\pprimecomp \andeq X \mapsto \Piet(X)\pprimecomp \cross \Piet(Y)\pprimecomp
	\end{equation*}
	define finitary functors \smash{$ \fromto{\Schqcqs_k}{\ProSpcpprime} $}, it suffices to prove the claim in the case that $ X $ is of finite type over $ k $ (\Cref{obs:equivalences_of_finitary_functors_can_be_checked_on_finitely_presented_S-schemes}).

	In this case, since regular schemes over algebraically closed fields are smooth, \Cref{rem:regular_v-hypercovers} shows that $ X $ admits a semi-simplicial \vhypercover $ U_{\bullet} $ by smooth $ k $-schemes.
	Using the facts that the \pprimecomplete étale homotopy type is a hypercomplete \vcosheaf (\Cref{thm:Pietprofin_satisfies_arc-descent}) and geometric realizations of semi-simplicial objects are universal in $ \ProSpcpprime $ \cite[Corollary 1.18]{Haine:profinite_completions_of_products}, we compute
	\begin{align*}
		\Piet(X \cross_k Y)\pprimecomp &\equivalent \colim_{[n] \in \Deltainjop} \Piet(U_n \cross_k Y)\pprimecomp && (\text{\vhyperdescent}) \\ 
		&\equivalent \colim_{[n] \in \Deltainjop} \paren{\Piet(U_n)\pprimecomp \cross \Piet(Y)\pprimecomp} && (\text{\Cref{ex:prime-to-p_Kunneth_formula_smooth_case}}) \\ 
		&\equivalent \paren{\colim_{[n] \in \Deltainjop} \Piet(U_n)\pprimecomp} \cross \Piet(Y)\pprimecomp && (\text{geometric realizations are universal}) \\ 
		&\equivalent \Piet(X)\pprimecomp \cross \Piet(Y) \period && (\text{\vhyperdescent}) \qedhere
	\end{align*}
\end{proof}

We now deduce some consequences of \Cref{thm:prime-to-p_Kunneth_formula}.
The first is the analogous Künneth formula for prime-to-$ p $ étale fundamental groups.

\begin{notation}
	Let $ \Sigma $ be a set of primes.
	Given a progroup $ G $, write $ G^{\Sigma} $ for the maximal \proSigma quotient of $ G $.
\end{notation}

\begin{recollection}[(fundamental groups of completions)]\label{rec:fundamental_groups_of_completions}
	Let $ \Sigma $ be a set of primes.
	Let $ U $ be a prospace that can be written as a \textit{finite} coproduct of connected prospaces.
	By \cite[Corollary 3.7]{MR0245577}, for any basepoint $ u \in U $, the natural map
	\begin{equation*}
		\fromto{\uppi_1(U,u)^{\Sigma}}{\uppi_1(U\Sigmacomp,u)}
	\end{equation*}
	is an isomorphism.

	As a consequence, if $ U $ is a \textit{profinite} space, then for any basepoint $ u \in U $, the natural map $ \fromto{\uppi_1(U,u)^{\Sigma}}{\uppi_1(U\Sigmacomp,u)} $ is also an isomorphism
\end{recollection}

\begin{corollary}
	Let $ k $ be a separably closed field of characteristic $ p \geq 0 $, and let $ X $ and $ Y $ be qcqs $ k $-schemes.
	Let $ \fromto{z}{X \cross_k Y} $ be a geometric point with images $ \fromto{x}{X} $ and $ \fromto{y}{Y} $.
	Then the natural continuous homomorphism
	\begin{equation*}
		\longfromto{\piet_1(X \cross_k Y,z)^{p'}}{\piet_1(X,x)^{p'} \cross \piet_1(Y,y)^{p'}}
	\end{equation*}
	is an isomorphism.
\end{corollary}

\begin{proof}
	Immediate from \Cref{thm:prime-to-p_Kunneth_formula,rec:fundamental_groups_of_completions}.
\end{proof}

\Cref{thm:prime-to-p_Kunneth_formula} also implies invariance results for the prime-to-$ p $ étale homotopy type.

\begin{example}[(invariance under extensions of separably closed fields)]\label{ex:invariance_of_the_etale_homotopy_type}
	Let $ k $ be a separably closed field of characteristic $ p \geq 0 $ and let $ X $ be a qcqs $ k $-scheme.
	Then for every extension of separably closed fields $ K \supset k $:
	\begin{enumerate}[label=\stlabel{ex:invariance_of_the_etale_homotopy_type}, ref=\arabic*]
		\item\label{ex:invariance_of_the_etale_homotopy_type.1} The natural map $ \fromto{\Piet(X_K)\pprimecomp}{\Piet(X)\pprimecomp} $ is an equivalence.

		\item\label{ex:invariance_of_the_etale_homotopy_type.2} The natural continuous homomorphism $ \fromto{\piet_1(X_K)^{p'}}{\piet_1(X)^{p'}} $ is an isomorphism.
	\end{enumerate}
\end{example}

\begin{remark}
	Item \enumref{ex:invariance_of_the_etale_homotopy_type}{2} recovers Landesman's recent invariance result \cite[Theorem 1.1]{arXiv:2005.09690}.
	See also \cite[Exposé XIII, Proposition 4.6]{MR50:7129}.
\end{remark}

\begin{corollary}[($ \AA^1 $-invariance)]\label{cor:A1-invariance}
	Let $ k $ be a separably closed field of characteristic $ p \geq 0 $.
	Then the functor
	\begin{equation*}
		\Piet(-)\pprimecomp \colon \fromto{\Schqcqs_k}{\ProSpcpprime}
	\end{equation*}
	is $ \AA^1 $-invariant.
\end{corollary}

\begin{proof}
	By the Künneth formula (\Cref{thm:prime-to-p_Kunneth_formula}), it suffices to show that \smash{$ \Piet(\AA_k^1)\pprimecomp \equivalent \pt $}.
	If $ p > 0 $, then since $ \AA_k^1 $ is smooth, connected, and affine, \cites[Proposition 15]{MR1383466}[Lemma 2.7(a)]{MR3549624} shows that
	\begin{equation*}
		\Pietprotrun(\AA_k^1) \equivalent \Kup(\piet_1(\AA_k^1),1) \period
	\end{equation*}
	The claim now follows from the fact that $ \piet_1(\AA_k^1) $ is a pro-$ p $ group.
	If $ p = 0 $, this follows from \Cref{ex:invariance_of_the_etale_homotopy_type}, the Riemann existence theorem, and the fact that the topological space \smash{$ \AA_{\CC}^1(\CC) $} is contractible.
\end{proof}


\subsection{Relative Künneth formulas}\label{subsec:relative_Kunneth_formulas}

Let $ k $ be a field and let $ X $ and $ Y $ be qcqs $ k $-schemes.
In this subsection, we prove Künneth formulas for the étale homotopy type of $ X \cross_k Y $, when $ k $ is \textit{not} separably closed.
The idea is to use the fundamental fiber sequence 
\begin{equation}\label{seq:fundamental_fiber}
	\begin{tikzcd}[sep=1.5em]
		\Pietprofin((X \cross_k Y)_{\kbar}) \arrow[r] & \Pietprofin(X \cross_k Y) \arrow[r] & \Pietprofin(\Spec(k)) 
	\end{tikzcd}
\end{equation}
of \cite{arXiv:2209.03476} to reduce to the separably closed case.

The next proposition is a general result that applies to a number of situations. 
Since the fiber sequence \eqref{seq:fundamental_fiber} need not remain a fiber sequence after completion away from $ \characteristic(k) $, some care is needed to formulate it.

\begin{recollection}\label{rec:Sigma-closed_field}
	Let $ \Sigma $ be a set of prime numbers, and write $ \Sigma' $ for the complement of $ \Sigma $ in the set of all primes.
	A field $ k $ is \defn{\Sigmaprimeclosed} in the sense of \cite[Definition 3.24]{arXiv:2209.03476} if for any separable closure $ \kbar \supset k $, the Galois group $ \Gal(\kbar/k) $ is a \proSigma group.
\end{recollection}

\begin{proposition}\label{prop:etale_homotopy_types_of_products_over_general_fields}
	Let $ k $ be a field with separable closure $ \kbar \supset k $, let $ X $ and $ Y $ be qcqs $ k $-schemes, and let $ \Sigma $ be a set of prime numbers.
	Assume the following conditions:
	\begin{enumerate}[label=\stlabel{prop:etale_homotopy_types_of_products_over_general_fields}, ref=\arabic*]
		\item\label{prop:etale_homotopy_types_of_products_over_general_fields.1} The field $ k $ is \Sigmaprimeclosed.

		\item\label{prop:etale_homotopy_types_of_products_over_general_fields.2} \emph{Künneth formula over $ \kbar $:} The natural map
		\begin{equation*}
			\longfromto{\Piet(\Xkbar \cross_{\kbar} \Ykbar)\Sigmacomp}{\Piet(\Xkbar)\Sigmacomp \cross \Piet(\Ykbar)\Sigmacomp}
		\end{equation*}
		is an equivalence.
	\end{enumerate}	
	Then the induced square 
	\begin{equation*}
		\begin{tikzcd}
			\Piet(X \cross_k Y)\Sigmacomp \arrow[r] \arrow[d] & \Piet(Y)\Sigmacomp \arrow[d] \\ 
			\Piet(X)\Sigmacomp \arrow[r] & \Piet(\Spec k)\Sigmacomp
		\end{tikzcd}
	\end{equation*}
	is a pullback square of profinite spaces.
\end{proposition}

\begin{proof}
	Write $ \Gup \colonequals \Gal(\kbar/k) $.
	Since $ k $ is \Sigmaprimeclosed, $ \Gup $ is a \proSigma group; hence the profinite space $ \BG $ is \Sigmacomplete.
	The choice of separable closure provides an identification $ \Piet(\Spec k) \equivalent \BG $.
	Since the basepoint $ \fromto{\pt}{\BG} $ is an effective epimorphism, it suffices to show that the natural map 
	\begin{equation}\label{eq:comparison_map_etale_homotopy_types_of_fields}
		\longfromto{\Piet(X \cross_{k} Y)\Sigmacomp}{\Piet(X)\Sigmacomp \crosslimits_{\BG} \Piet(Y)\Sigmacomp}
	\end{equation}
	becomes an equivalence after pullback along $ \fromto{\pt}{\BG} $.
	
	First we compute the fiber of the left-hand side of \eqref{eq:comparison_map_etale_homotopy_types_of_fields} over $ \BG $.
	To do this, note that the fundamental fiber sequence \cite[Corollary 3.28]{arXiv:2209.03476} implies that the natural square
	\begin{equation*}
		\begin{tikzcd}
			\Piet((X \cross_{k} Y)_{\kbar})\Sigmacomp \arrow[d] \arrow[r] & \Piet(X \cross_{k} Y)\Sigmacomp \arrow[d] \\
			\pt \arrow[r] & \BG 
		\end{tikzcd}
	\end{equation*}
	is a pullback.
	Moreover, combining \Cref{obs:basechange_of_product_is_product_of_basechange} with assumption \enumref{prop:etale_homotopy_types_of_products_over_general_fields}{2} shows that 
	\begin{align*}
		\Piet((X \cross_{k} Y)_{\kbar})\Sigmacomp &\equivalent \Piet(\Xkbar \cross_{\kbar} \Ykbar)\Sigmacomp \\
		&\equivalent \Piet(\Xkbar)\Sigmacomp \cross \Piet(\Ykbar)\Sigmacomp
	\end{align*}

	To compute the fiber of the right-hand side of \eqref{eq:comparison_map_etale_homotopy_types_of_fields} over $ \BG $, consider the cube
	\begin{equation}\label{cube:etale_homotopy_types_geometric_fibers}
        \begin{tikzcd}[column sep={12ex,between origins}, row sep={8ex,between origins}]
            \Piet(\Xkbar)\Sigmacomp \cross \Piet(\Ykbar)\Sigmacomp \arrow[rr] \arrow[dd]  \arrow[dr] & & \Piet(\Ykbar)\Sigmacomp \arrow[dd]  \arrow[dr] \\
            & \displaystyle\Piet(X)\Sigmacomp \crosslimits_{\BG} \Piet(Y)\Sigmacomp \arrow[rr, crossing over] & & \Piet(Y)\Sigmacomp \arrow[dd]  \\
            \Piet(\Xkbar)\Sigmacomp \arrow[rr] \arrow[dr] & & \pt \arrow[dr] \\
            & \Piet(X)\Sigmacomp \arrow[rr] \arrow[from=uu, crossing over] & & \BG \period & 
        \end{tikzcd}
    \end{equation}
    Again by the fundamental fiber sequence, we see that the rightmost vertical and bottom horizontal faces are pullback squares.
    Since the front and back vertical faces of \eqref{cube:etale_homotopy_types_geometric_fibers} are by definition pullbacks, we see that all squares appearing in \eqref{cube:etale_homotopy_types_geometric_fibers} are pullback squares.
    In particular, 
    \begin{equation*}
    	\paren{\Piet(X)\Sigmacomp \crosslimits_{\BG} \Piet(Y)\Sigmacomp} \crosslimits_{\BG} \pt 
    	\equivalent \Piet(\Xkbar)\Sigmacomp \cross \Piet(\Ykbar)\Sigmacomp \period
    \end{equation*}
    Thus the natural map \eqref{eq:comparison_map_etale_homotopy_types_of_fields} induces an equivalence on fibers, as desired.
\end{proof}

\begin{corollary}[(relative Künneth formula, proper case)]\label{cor:relative_Kunneth_formula_proper}
	Let $ k $ be a field and let $ X $ and $ Y $ be qcqs $ k $-schemes.
	If $ Y $ is proper over $ k $, then the induced square
	\begin{equation*}
		\begin{tikzcd}
			\Pietprofin(X \cross_k Y) \arrow[r] \arrow[d] & \Pietprofin(Y) \arrow[d] \\ 
			\Pietprofin(X) \arrow[r] & \Pietprofin(\Spec k)
		\end{tikzcd}
	\end{equation*}
	is a pullback.
\end{corollary}

\begin{proof}
	Apply \Cref{prop:etale_homotopy_types_of_products_over_general_fields} for $\Sigma$ the set of all primes; hypothesis \enumref{prop:etale_homotopy_types_of_products_over_general_fields}{1} is trivially satisfied and \Cref{ex:Kunneth_formula_proper_case} shows that hypothesis \enumref{prop:etale_homotopy_types_of_products_over_general_fields}{2} is satisfied.
\end{proof}

\begin{corollary}[(prime-to-$ p $ relative Künneth formula)]\label{cor:relative_Kunneth_formula_prime-to-p}
	Let $ k $ be a field of characteristic $ p \geq 0 $ and let $ X $ and $ Y $ be qcqs $ k $-schemes.
	If $ k $ is \pclosed, then the induced square
	\begin{equation*}
		\begin{tikzcd}
			\Piet(X \cross_k Y)\pprimecomp \arrow[r] \arrow[d] & \Piet(Y)\pprimecomp \arrow[d] \\ 
			\Piet(X)\pprimecomp \arrow[r] & \Piet(\Spec k)\pprimecomp
		\end{tikzcd}
	\end{equation*}
	is a pullback.
\end{corollary}

\begin{proof}
	Apply \Cref{prop:etale_homotopy_types_of_products_over_general_fields}; \Cref{thm:prime-to-p_Kunneth_formula} shows that hypothesis \enumref{prop:etale_homotopy_types_of_products_over_general_fields}{2} is satisfied.
\end{proof}

\begin{warning}
	If $ k $ is not \pclosed, \Cref{cor:relative_Kunneth_formula_prime-to-p} is false.
	See \cite[Warning 3.23]{arXiv:2209.03476}.
\end{warning}

We conclude with two remarks about how to use \Cref{cor:relative_Kunneth_formula_proper,cor:relative_Kunneth_formula_prime-to-p} to deduce Künneth formulas for symmetric powers.
These \textit{symmetric Künneth formulas} are analagous to Deligne's results about the étale cohomology of symmetric powers \cite[Exposé XVII, Théorème 5.5.21]{MR50:7132}; see the introduction of \cite{arXiv:1810.00351} for a summary of how these results differ from Deligne's.

\begin{remark}[(symmetric Künneth formula, proper case)]\label{rem:symmetric_Kunneth_formula_proper_case}
	Let $ k $ be a field and $ X $ a proper $ k $-scheme.
	Following ideas of Hoyois \cite[\S5]{arXiv:1810.00351}, Chough proved that if $ k $ is separably closed, there is a Künneth formula for symmetric powers
	\begin{equation*}
		\equivto{\Pietprofin(\Sym^n X)}{\Sym^n \Pietprofin(X)}
	\end{equation*}
	\cite[Theorem 6.12]{MR4493612}.
	The only part of Chough's proof that uses that the ground field is separably closed is the Künneth formula (which, at the time of Chough's paper, was only known over separably closed fields).
	As a consequence of \Cref{cor:relative_Kunneth_formula_proper}, Chough's proof shows that the symmetric Künneth formula for proper schemes holds over arbitrary base fields.
\end{remark}

\begin{remark}[(prime-to-$ p $ symmetric Künneth formula)]\label{rem:symmetric_Kunneth_formula_prime-to-p}
	Let $ k $ be a field of characteristic $ p \geq 0 $ and let $ X $ be a quasiprojective $ k $-scheme.
	If $ k $ is separably closed, Hoyois proved that for any prime $ \el \neq p $, the natural map 
	\begin{equation*}
		\fromto{\Piet(\Sym^n X)}{\Sym^n \Piet(X)}
	\end{equation*}
	becomes an equivalence after $ \ZZ/\el $-homological localization \cite[Theorem 5.6]{arXiv:1810.00351}.
	Similarly to \Cref{rem:symmetric_Kunneth_formula_proper_case}, the key `non-formal' input is that the natural map 
	\begin{equation*}
		\fromto{\Piet(X^{\cross n})}{\Piet(X)^{\cross n}}
	\end{equation*}
	becomes an equivalence after $ \ZZ/\el $-homological localization \cite[Proposition 5.1]{arXiv:1810.00351}.
	Since we now know the stronger Künneth formula \Cref{cor:relative_Kunneth_formula_prime-to-p}, Hoyois' proof shows that if $ k $ is a $ p $-closed field, then the natural map
	\begin{equation*}
		\fromto{(\Piet(\Sym^n X))\pprimecomp}{(\Sym^n \Piet(X))\pprimecomp}
	\end{equation*}
	of profinite spaces over $ \Pietprofin(\Spec(k)) $ is an equivalence.
\end{remark}